\documentclass[a4paper,12pt]{article}
\usepackage{authblk}
\usepackage{fullpage}
\usepackage{ae,lmodern}
\usepackage[english,french]{babel}
\usepackage[utf8]{inputenc}  
\usepackage[T1]{fontenc}
\usepackage{url,csquotes}
\usepackage{float}
\usepackage{amsfonts,amssymb,enumerate}
\usepackage{amsmath}
\usepackage[shortlabels]{enumitem}
\allowdisplaybreaks[1]
\usepackage{amsthm}
\usepackage{graphicx}
\usepackage{bbm}
\usepackage{listings}
\usepackage{titling}
\usepackage{dsfont}
\usepackage{color}
\usepackage{mathrsfs}
\usepackage{enumitem}
\usepackage{bmpsize}
\usepackage{multibib}
\usepackage[hidelinks,hyperfootnotes=false]{hyperref}
\newcommand{\E}{\mathbb{E}}

    \newcommand{\Prb}{\mathbb{P}}
	
		\newcommand{\cA}{\mathcal{A}}
		\newcommand{\cO}{\mathcal{O}}
		\newcommand{\cE}{\mathcal{E}}
		\newcommand{\cF}{\mathcal{F}}
		
	\newcommand{\cB}{\mathcal{B}}

			\newcommand{\cP}{\mathcal{P}}
	\newcommand{\sR}{\mathbb{R}}
		\newcommand{\dis}{\mathrm{d}}

						\newcommand{\rS}{\mathrm{S}}
	\newcommand{\sN}{\mathbb{N}}

		\newcommand{\rB}{\mathrm{B}}
	\newcommand{\rL}{\mathrm{L}}

				\DeclareMathOperator{\Inf}{Inf}

				\DeclareMathOperator{\Var}{Var}

    \newcommand{\sZ}{\mathbb{Z}}
    \newcommand{\sC}{\mathscr{C}}
    \newcommand{\ep}{\varepsilon}

    \newcommand{\ind}{\mathbf{1}}

 \theoremstyle{plain}   
 \newtheorem{thm}{Theorem}[section]
\newtheorem{lem}[thm]{Lemma}
\newtheorem{cor}[thm]{Corollary}
\newtheorem{prop}[thm]{Proposition}
\newtheorem{rk}[thm]{Remark}

\newtheorem{claim}[thm]{Claim}

\numberwithin{equation}{section}

\title{Almost sharp sharpness for Poisson Boolean percolation}
\date{}

 \author{Barbara Dembin}

 \author{Vincent Tassion}
 \affil[1]{D-MATH, ETH Z\"urich,
    Switzerland.}

\begin{document}

\selectlanguage{english}
\maketitle
\begin{abstract} We consider Poisson Boolean percolation on $\sR^d$ with power-law distribution on the radius with a finite $d$-moment for $d\ge 2$. We prove that subcritical sharpness occurs for all but a countable number of power-law distributions. This extends the results of Duminil-Copin--Raoufi--Tassion \cite{DCRT} where subcritical sharpness is proved under the assumption that the radii distribution has a $5d-3$ finite moment. Our proofs techniques are different from \cite{DCRT}: we do not use randomized algorithm and rely on specific independence properties of Boolean percolation, inherited from the underlying Poisson process.

 
 We also prove supercritical sharpness for any distribution with a finite $d$-moment and the continuity of the critical parameter for the truncated distribution when the truncation goes to infinity.
\end{abstract}
\section{Introduction}
\paragraph{Overview}

Boolean percolation was introduced by Gilbert in \cite{gilbert} as a continuous version of  Bernoulli percolation, introduced by Broadbent and Hammersley \cite{BroadbentHammersley}.  We consider a Poisson point process of intensity $\lambda>0$ on $\sR^d$ and on each point, we center a ball of potentially random radius. In Boolean percolation we are interested in the connectivity properties of the occupied set: it is defined as the subset of $\mathbb R^d$ consisting of all the points covered by at least one ball. 


This model undergoes a phase transition in $\lambda$ for the existence of an unbounded connected component of balls. For $\lambda<\lambda_c$, all the connected components are bounded, and for $\lambda>\lambda_c$, there exists at least one unbounded connected component.  

When the radii of the balls are bounded, Boolean percolation exhibits a very similar behaviour as standard Bernoulli percolation. In contrast, when the law of the radii has fat tails, the existence of large balls creates long-range dependencies in the occupied set, leading to different behaviours, and proving them often requires new methods. This is particularly true for the subcritical  phase: When the radii are bounded, the connection probability between two fixed points of $\mathbb R^d$ decays exponentially fast in the distance between the two points: there is subcritical sharpness.  But this cannot be the case when the radii have subexponential tails, because the connection probability is always larger than the probability that the two points are covered by a large ball. Nevertheless, a notion of sharpness can be defined in this case. This was established in \cite{DCRT} under a moment assumption. In this paper, we extend the result to a minimal moment assumptions, but we restrict to a particular family of laws for the radii. We also establish  supercritical sharpness, by relying on standard methods for Bernoulli percolation.
\paragraph{Definition of the model}
Let us more formally define Boolean percolation. Let $d\geq 2$. Denote by $\|\cdot\|$ the $\ell_2$-norm on $\mathbb R^d$. For $r>0$ and $x\in\sR^d$, set \[\mathrm B^x_r:=\left\{y\in\sR^d:\,\|y-x\|\leq r\right\}\qquad\text{and}\qquad\partial  \mathrm B^x_r:=\left\{y\in\sR^d:\,\|y-x\|= r\right\}\] for the closed ball of radius $r$ centered at $x$ and its boundary. For short, we will write $\mathrm B_r$ for $\mathrm B_ r^0$.
For a subset $\eta$ of $ \sR^d\times\sR_+$, we define the occupied set associated to $\eta$ by
\[\cO(\eta):=\bigcup_{(z,r)\in\eta}\mathrm B_r^z.\]
Let $\mu$ be a measure on $\sR_+$ and $\lambda>0$. Let $\eta$ be a Poisson point process of intensity $\lambda\, dz\otimes \mu$ where $dz$ is the Lebesgue measure on $\sR^d$. Write $\Prb_{\lambda,\mu}$ for the law of $\eta$ and $\E_{\lambda,\mu}$ for the expectation under the law $\Prb_{\lambda,\mu}$. We will work with measures $\mu$ such that
\begin{align}\label{cond:mu}
\int_{\sR_+}t^dd\mu(t)<\infty.
\end{align}
This hypothesis is natural for the study of percolation properties of the occupied set $\mathcal O(\eta)$. Indeed, Hall proved in \cite{hall} that this condition is necessary and sufficient to avoid that all the space is covered. 
We say that two points $x$ and $y$ in $\sR^d$ are connected by $\eta$ if there exists a continuous path in $\cO (\eta)$ that joins $x$ to $y$. We say that two sets $A$ and $B$ are connected if there exists $x\in A$ and $y\in B$ such that $x$ and $y$ are connected by $\eta$. We denote by $\{A\longleftrightarrow B\}$ this event.

The phase transition of Boolean percolation is defined as follows. 
Define the critical parameter by
$$\lambda_c(\mu):=\sup\left\{\lambda\geq 0: \lim_{r\rightarrow \infty}\Prb_{\lambda,\mu}\left( 0\longleftrightarrow \partial \mathrm B_r\right)=0\right\}.$$
Under the minimal assumption \eqref{cond:mu}, Gou\'er\'e proved in \cite{gouere} that $0<\lambda_c(\mu)<\infty$.
\paragraph{Subcritical sharpness}


For Bernoulli percolation, the terminology "subcritical sharpness" refers to the exponential decay of the connection probabilities in the subcritical regime, see \cite{menshikov86,AizenmanBarsky87,D-CT}. For Boolean percolation with subexponential radii distribution, such exponential decay does not hold for any $\lambda>0$. In order to define subcritical sharpness, we introduce the following critical point (following Gouéré--Théret \cite{GouereTheret})
 \[\widehat \lambda_c(\mu):=\inf\left\{\lambda\geq 0:\,\inf_{r>0}\Prb_{\lambda,\mu}(\mathrm  B_r\longleftrightarrow\partial \mathrm B_{2r})>0\right\}.\]
In the regime $\lambda<\widehat{\lambda}_c(\mu)$, renormalization arguments apply and provide an accurate description of the model, see Gouéré,Gouéré--Théret \cite{gouere,GouereTheret}. For example, when $\mu$ has a subexponential tail, it was proved in \cite{DCRT} that the probability that $0$ is connected to distance $n$  is equivalent to the probability that the origin is covered by one ball intersecting $\partial B_n$. 
\emph{Subcritical sharpness} corresponds to coincidence of the two critical points $\widehat \lambda_c(\mu)=\lambda_c(\mu)$

 
 In the case of bounded radius, Boolean percolation behaves at a high level in a similar way as standard percolation. The subcritical sharpness is known, and there is exponential decay in the subcritical regime \cite{ziesche,DCRT, faggionatohlafo}. In the case of unbounded radius, the subcritical sharpness has been proved by Duminil-Copin--Raoufi--Tassion \cite{DCRT}, when the distribution has a $5d-3$ finite moment, using randomized algorithms and OSSS inequality (they made a discretization of the space-radius to be able to use the OSSS inequality). Later  Last--Pecatti--Yogeshwaran  \cite{OSSScontinuous}  proved a continuous version of the OSSS inequality \cite{OSSS} and obtained sharpness for a general class of models, but their result does not include all the Boolean percolation with the minimal assumption \eqref{cond:mu}.  A key idea in the OSSS approach is to control the probability that a given ball is revealed when certain  exploration algorithms are run. This method fails for fat tails, because the revealment of large balls by the randomized algorithm is too large, and new ideas are needed to bypass this obstacle.

 In the special case of $d=2$, the subcritical sharpness has been proved by Ahlberg--Tassion--Teixeira \cite{ATT18} under minimal assumption \eqref{cond:mu} using crucially the planarity (the proofs relies on Russo--Seymour--Welsh theory).
 

In this paper, we investigate the subcritical sharpness for distribution $\mu$ under the minimal assumption \eqref{cond:mu}. We restrict our study to the following family of power-law distributions.
For $\delta>0$, set 
\begin{equation*}
  \mu_\delta:=\frac{1}{r^{d+1+\delta}}\ind_{r\geq 1}dr.  
\end{equation*}
The main result of this paper is the following.
\begin{thm}[Almost sharp subcritical sharpness]\label{thm:main}There exists $\mathcal D\subset(0,\infty)$ at most countable such that for all $\delta \in (0,\infty)\setminus \mathcal D$, we have 
\begin{equation}\label{star3}
\lambda_c(\mu_\delta)=\widehat \lambda_c(\mu_\delta).
\end{equation}
\end{thm}
\begin{rk}
The set $\mathcal D$  is defined as the set of discontinuity points of $\delta\mapsto\lambda_c(\mu_\delta)$. By monotonicity, we know that $\mathcal D$ is at most countable, but we expect that $\lambda_c(\delta)$ is actually continuous in $\delta$, namely that $\mathcal D=\emptyset$, but we were not able to prove this. 
\end{rk}

\begin{rk}We call this result almost sharp sharpness, because it states that for the laws $\{\mu_\delta, \delta\}$, Boolean percolation satisfies subcritical sharpness under the minimal assumption \eqref{cond:mu} (except possibly for countably many $\delta$). In particular, it implies that for  every $d_0> d$, there exists a distribution with infinite $d_0$-moment for which subcritical sharpness holds. 
\end{rk}

Our proof techniques are very different from \cite{DCRT}, in particular we do not use randomized algorithm. We prove a Talagrand formula \cite{Talagrand} for Boolean percolation to prove a sharp threshold. There is a real difficulty in applying Talagrand formula in this context. The big balls play the role of a dictator so we need to increase their intensity a lot to prove a sharp threshold. Doing a small sprinkling in $\lambda$ is not sufficient to obtain a sharp threshold in that context.

\paragraph{Supercritical sharpness}
 In the supercritical regime of bond percolation, the key property to start a renormalization is the so-called local uniqueness: with high probability a large box contains a unique macroscopic cluster. If this property is true for every $p>p_c$, we say that the model has supercritical sharpness. The main step to prove local uniqueness is to prove percolation in slabs using Grimmett--Marstrand argument \cite{GM}. They proved that for any $p>p_c$, for $k$ large enough there exists an infinite connected component in the slab $S_k:=\sR^2 \times [-k,k]^{d-2}$. 
For Bernoulli  bond percolation on $\mathbb Z^d$, supercritical sharpness is more involved and delicate than its subcritical counterpart. In contrast, for Boolean percolation, subcritical appears more challenging than supercritical percolation: as discussed in the previous section, establishing subcritical sharpness requires new ideas, while for supercritical sharpness,  the proof follows essentially the same lines as in the case of bond percolation. Some small adaptations are needed, but the proof is rather simpler for Boolean than the lattice version, thanks to rotational symmetries specific to a continuous model. In this paper, we prove that supercritical sharpness holds for any distribution satisfying \eqref{cond:mu}.
We define
\[\lambda_c^{S_k}(\mu):=\sup\left\{\lambda\geq 0: \lim_{r\rightarrow\infty}\Prb_{\lambda,\mu}\left(0\stackrel {\cO(\eta\cap (S_k\times \sR_+))}{\longleftrightarrow }\partial\mathrm B_r\right)=0\right\}.\]
The following theorem is proved using the seedless renormalization scheme of \cite{DCKT}, which is a refinement of the original argument \cite{GM}.
\begin{thm}\label{thm:main2}Let $\mu$ be a distribution that satisfies \eqref{cond:mu}. We have
\[\inf_{n\ge 1}\lambda_c^{S_n}(\mu|_{[0,n]})=\lambda_c(\mu)\,\]
where $\mu|_{[0,n]}$ is the restriction of the distribution on the interval $[0,n]$.
\end{thm}
We deduce from this result the two following straightforward corollaries. We have supercritical sharpness for any distribution $\mu$.
\begin{cor}[Supercritical sharpness]\label{cor:1} Let $\mu$ be a distribution that satisfies \eqref{cond:mu}. We have
\[\lim_{k\rightarrow\infty}\lambda_c^{S_k}(\mu)=\lambda_c(\mu).\]
\end{cor}
The following corollary states that the function $\lambda_c$ is continuous with respect to the truncated distribution.
\begin{cor}[Continuity of $\lambda_c$ for truncated radius]\label{cor:2} Let $\mu$ be a distribution that satisfies \eqref{cond:mu}. We have
\begin{equation*}
    \lim_{n\rightarrow\infty}\lambda_c(\mu|_{[0,n]})=\lambda_c(\mu).
\end{equation*}
\end{cor}

\paragraph{Key steps and inputs for the proof Theorem \ref{thm:main}.}
We proceed by contradiction and assume that there exists a non empty regime $\widehat\lambda_c(\delta)<\lambda_c(\delta)$. In other words, we assume the existence of a regime $\lambda\in (\widehat\lambda_c(\delta),\lambda_c(\delta))$ for which there is not infinite cluster but the crossing probability $\mathbb P_\lambda(\rB_n\longleftrightarrow \partial \rB_{2n})$ does not go to $0$ (along an infinite  subsequence of $n$). We consider $\lambda\in (\widehat\lambda_c(\delta),\lambda_c(\delta))$, and we prove that for arbitrarily small $\varepsilon >0$, the process percolates at parameter $\lambda+\varepsilon$, which is a contradiction to $\lambda<\lambda_c(\delta)$. We proceed in 3 steps:

\begin{description}
\item[Step 1: sharpness for truncated radii.] If $\lambda>\widehat \lambda_c$, the cluster of $\rB_n$  cannot be too small: at least for some scales, it must intersect $\partial \rB_{2n}$ with positive probability. By using a union bound, we can deduce a good lower bound on the probability of $\rB_n\longleftrightarrow\partial \rB_N$ for $N\gg n$. On the other hand, if $\lambda <\lambda_c$ we know that the set of balls of radius $\le n$ do not percolate. Furthermore, by the known sharpness for bounded radii, the connection probabilities for the truncated process (obtained be keeping only the balls of radius  $\le n$)  decay exponentially fast. Hence $\rB_n$ cannot be connected too far without using large balls, therefore the cluster of $\rB_n$ must contain balls of radius $\rB_r$ for $r\ge n$. We can make this idea quantitative by using the $\varphi_p(S)$ argument of \cite{D-CT} (developed in \cite{ziesche} for Boolean percolation): For arbitrary $\alpha>0$, we obtain,
\begin{equation}
    \mathbb P_\lambda(\mathcal E_n)\ge \frac 1 {n^{\alpha}},
\end{equation}
where $\cE_{n}:= \left\{\text{$\mathrm B_n$ is connected  to an open ball of radius larger than $n$}\right\}$ (see figure \ref{fig}). 
\item[Step 2: sharp threshold in $\delta$ via hypercontractivity.]  The event $\mathcal E_n$ is rotation invariant, and we expect such event to satisfy a sharp threshold. (A small increase in the parameter is sufficient to increase the probability close to one. For a long time in the project, we did not think that this part would be difficult, but we faced two obstacles. The first one is conceptual: there is no general sharp threshold in $\lambda$ (see Section \ref{sec:example} for an example that illustrates this fact). We overcome this difficulty by proving a sharp threshold in $\delta$ instead: when varying $\delta$, the density of large balls increases much faster than the density of the small ones, which prevents the large balls to play a "dictator" role on a full range of parameters. The second difficulty comes from the fact that we want to prove a sharp threshold for rotational invariant events. A natural strategy would be to discretize the space and the radius and apply a discrete Talagrand formula. However, there is a theoretical obstacle in pursuing this strategy: for dimensions $d\ge 3$, there is no rotational invariant way to discretize the space. This is a consequence of the fact that the group of orthogonal transformations is not abelian in dimension $d\ge 3$. Hence, discretizing the model implies losing all the symmetries of the continuous one. For this reason, we had to prove a continuous version of Talagrand's formula, by using a suitable encoding and limiting procedure. By applying this formula to the event $\mathcal E_n$, we obtain that
\begin{equation}
    \mathbb P_{\lambda,\delta-\eta}(\mathcal E_n)\ge 1-\frac 1{n^\alpha}.
\end{equation}

\item[Step 3: Grimmett--Marstrand dynamic exploration:] thanks to the previous step we know that the event $\cE_n$ occurs with high probability for the distribution $\mu_{\delta-\ep}$.  The open ball in the definition of the event $\cE_n$ plays the role of the seed in Grimmett--Marstrand dynamic renormalization. Up to a small sprinkling $\varepsilon$ in $\lambda$, we can prove that there exists an infinite connected component with positive probability. It follows that \begin{equation}\label{eq:contradiction}\lambda+\varepsilon>\lambda_c(\mu_{\delta-\eta}).\end{equation}

\item[Conclusion.] Equation~\eqref{eq:contradiction} is not exactly the contradiction we were aiming for because we had to reduce $\delta$, nevertheless we still reach a contradiction if $\delta$ is a continuity point of $\lambda_c$.
\end{description}



\begin{center}
\begin{figure}[H]
\def\svgwidth{0.15\textwidth}
\hspace{6cm}
 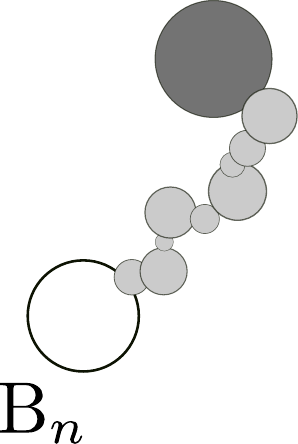
 \caption[figen]{\label{fig}Illustration of the event $\cE_{n}$}
\end{figure}
\end{center}

\paragraph{Organization of the paper} In the next section, we state the three key propositions corresponding to the three steps and prove Theorem \ref{thm:main}.  In Section \ref{sec:equalitylc}, we prove that in the fidicious regime, with not too small probability, balls are connected to larger open balls (this corresponds to Step 1). In Section \ref{sec:ST}, we prove a Talagrand formula for Boolean percolation and prove a sharp threshold result for rotational invariant events (this corresponds to Step 2 of the sketch of the proof). In Section \ref{sec:GM}, we prove that there is percolation up to a sprinkling in the intensity using a Grimmett--Marstrand dynamic exploration using big balls as seeds (this corresponds to Step 3). Finally, in Section \ref{sec:main2}, we prove Theorem \ref{thm:main2}.

\paragraph{Acknowledgements}
We thank Tobias Hartnick and Matthis Lehmk\"uhler for useful exchanges regarding discretization of the sphere. We thank Jean-Baptiste Gouéré for his careful reading and his useful comments.
This project has received funding from the European Research Council (ERC) under the European Union’s Horizon 2020 research and innovation program (grant agreement No
851565). 
\section{Proof of Theorem \ref{thm:main}}\label{sec:main}
In this section, we state the main propositions corresponding to the three steps of the sketch of proof and prove Theorem \ref{thm:main}.

For short, we will write $\lambda_c(\delta), \widehat\lambda_c(\delta)$ instead of $\lambda_c(\mu_\delta), \widehat\lambda_c(\mu_\delta)$ and $\mathbb P_{\lambda,\delta}$ instead of $\mathbb P_{\lambda,\mu_\delta}$.

For $\rho\ge 1$, we define
\[\cE_{n,N}(\rho):=\left\{\exists (x,r)\in \eta \cap((\mathrm B_{(\rho+1/2)N}\setminus \mathrm B_{\rho N})\times[n,+\infty)):  \mathrm B_n\stackrel{\cO(\eta\cap(\mathrm B_{(\rho+1/2)N} \times \sR))}{ \longleftrightarrow} \mathrm B_r^ x\right\}.\]
For technical reasons arising in the proof of Grimmett--Marstrand, we need this extra parameter $\rho$.
In the regime $(\widehat \lambda_c(\delta), \lambda_c(\delta))$, we can prove that there exists $n\le N$ such that for any $\rho \in[1,2d]$, the event $\cE_{n,N}(\rho)$ occurs with a not too small probability. This corresponds to Step 1 of the sketch of the proof (it is proved in Section \ref{sec:equalitylc}).
\begin{prop}\label{prop:seed}Let $\delta>0$. Assume that $\widehat\lambda_c(\delta)<\lambda_c(\delta)$. Let $\ep>0$ and $\lambda\in(\widehat \lambda_c(\delta),\lambda_c(\delta))$. For any $n_0\ge 1$, there exists  $n\ge n_0$ and $N\ge 2n $ such that 
\begin{equation*}
   \forall \rho \in[1,2d]\qquad\Prb_{\lambda,\delta}(\cE_{n,N}(\rho))\ge \frac{1}{n^ \ep}.
\end{equation*}
\end{prop}


In the second step of the proof, we prove a sharp threshold in $\delta$.  In Section \ref{sec:ST}, we prove an inhomogeneous Talagrand formula for Boolean percolation in Proposition \ref{prop:talpoisson}. The following proposition is proved in greater generality in Section \ref{sec:ST}.

 \begin{prop}\label{prop:ST}Let $\lambda,\delta>0$. There exists $\kappa>0$ such that the following holds. 
For every $\varepsilon\in (0,\delta/2\kappa)$, $n\ge \kappa$, $N\ge 2n$ and $\rho\in [1,2d]$,
\begin{equation*}
     \Prb_{\lambda,{\delta}}(\cE_{n,N}(\rho))\ge \frac{1 }{n^ \ep}\implies  \Prb_{\lambda,{\delta -\kappa \varepsilon}}(\cE_{n,N}(\rho))\ge 1- \frac{1 }{n^ \ep}.
 \end{equation*}
\end{prop}

Once we know that a big ball is connected to another open big ball with a very good probability, this big ball will play the role of the seed in Grimmett--Marstrand proof using dynamic renormalization. The following proposition (proved in Section \ref{sec:GM}) is a Grimmett--Marstrand type proposition that corresponds to the Step 3 of the sketch of proof.
\begin{prop}\label{prop:GM}Let $\mu$ be a distribution with a finite $d$-moment. There exists $\ep_0>0$ depending only on $d$ such that for all $\ep\le \ep_0$, $n\le N$, if we have
for every $\rho \in[1,2d]$
\begin{equation*}
    \Prb_{\lambda,\mu}(\cE_{n,N}(\rho))\ge 1-\ep
\end{equation*}
then $\lambda>\lambda_c(\mu)-\frac{\lambda K_d}{|\log \ep|}$ where $K_d>0$ depends only on $d$.
\end{prop}
Now we have all the ingredients to prove Theorem \ref{thm:main}.
\begin{proof}[Proof of Theorem \ref{thm:main}]
Let $\delta>0$ such that $\widehat\lambda_c(\delta)<\lambda_c(\delta)$. Let us prove that $\delta \in\mathcal D$. 
Let $\lambda \in (\widehat\lambda_c(\delta),\lambda_c(\delta)) $. Let $\kappa$ be as in Proposition \ref{prop:ST}.
 Let $\ep\in(0,\delta/2\kappa)$ and $n_0\ge \kappa$. By Proposition \ref{prop:seed}, there exist $n\ge n_0$ and $N\ge 2n $ such 
\begin{equation*}
  \forall \rho \in[1,2d]\qquad  \Prb_{\lambda,\delta}(\cE_{n,N}(\rho))\ge \frac{1}{n^ \ep}.
\end{equation*}
By Proposition \ref{prop:ST}, we obtain 
\begin{equation*}
  \forall \rho \in[1,2d]\qquad    \Prb_{\lambda,\delta-\kappa \ep}(\cE_{n,N}(\rho))\ge 1-\frac{1}{n^ \ep }.
\end{equation*}
By Proposition \ref{prop:GM}, it follows that
\[\lambda\ge\lambda_c(\delta-\kappa\ep)-\frac{ {\lambda}_c(\delta) K_d}{\ep \log n}.\]
There exists an infinite number of $n$ such that the previous inequality holds.
By letting first $n$ go to infinity and then $\ep$ go to $0$, we obtain that
\[\lambda\ge \lambda_c(\delta^-)\,\]
where $\lambda_c(\delta^-)$ is the left limit of the function $\delta\mapsto\lambda_c(\delta) $ at $\delta$.
We recall that $\lambda<\lambda_c(\delta)$, it yields that 
\[\lambda_c(\delta^-)<\lambda_c(\delta)\]
and $\delta \in\mathcal D$.
\end{proof}

\section{Connection to large balls in the fidicious regime}\label{sec:equalitylc}

In this section, we prove Proposition \ref{prop:seed}.
\subsection{Correlation length for truncated radii}

In this section we consider the law of radii
\begin{equation}
    \mu=\mu_\delta|_{[0,n]},
\end{equation}
and we write $\mathbb P_\lambda^{n}=\mathbb P_{\lambda,\mu}$ for the corresponding Boolean percolation measure, and $\mathbb E_\lambda^n$ for the corresponding expectation. For bounded radii, sharpness of the phase transition is known. In particular, one can define a correlation length in the subcritical regime, and  we collect here some quantitative statements that can be deduced from previously known arguments.

Denote by $\cB(\sR^ d)$ the Borel $\sigma$-algebra of $\sR^ d$.
Following Ziesche \cite{ziesche} (see also \cite{D-CT}), define for every $\lambda\ge0$, $n\ge1$ and  $S\in\cB(\sR^d)$ such that $\mathrm B_n\subset S$
\begin{equation}
    \varphi_{\lambda}^{n}(S):=\mathbb E_\lambda^n\big[\sum_{(z,r)\in\eta} \mathbf1[\mathrm B_r^z \cap\partial S \neq \emptyset, \mathrm B_n\stackrel{\{(w,s)\in\eta: \mathrm B_s^w\subset S\})}{\longleftrightarrow }\mathrm B_r^z  \big]
\end{equation}


This functional is instrumental in the proof of sharpness, and we will use the following properties. Let $\lambda\ge 0$,  $N\ge n\ge 1$. 
\begin{enumerate}[\bf P1]
    \item\label{item:1}  If there exists  $ S\in\cB(\sR^d)$ such that $\mathrm B_n \subset S\subset\mathrm B_N$ and $\varphi ^n_{\lambda}(S)\le 1/e$ then 
\[\forall \ell \ge N \qquad \Prb_{\lambda}^n(\mathrm B_n\longleftrightarrow\partial \mathrm B_\ell)\le \exp\left(-\left\lfloor{\ell}/{N}\right\rfloor\right).\]  

\item\label{item:2} If for all $ S\in\cB(\sR^d)$ such that $\mathrm B_n \subset S\subset \mathrm B_N$ we have $\varphi ^n_{\lambda}(S)\ge 1/e$ then
\[\forall \lambda'>\lambda\qquad \Prb_{\lambda'}^n(\mathrm B _n \longleftrightarrow\partial \mathrm B_N)\ge \frac{\lambda'-\lambda}{e\lambda'}.\]
\end{enumerate}

These properties are not exactly stated in \cite{ziesche} but they can be obtained from the proof of the first and third item of Theorem 3.1 in \cite{ziesche}. These properties give rise to a natural definition of a correlation length in the subcritical regime: For every $n\ge 1$ and every $\lambda\ge 0$,
let 
\begin{equation*}
   \mathrm L_\lambda(n):=\inf\left\{\ell\in[n,+\infty): \exists S\in\cB(\sR^d),\,\mathrm B_n\subset S\subset\mathrm B_{\ell},\,\varphi_\lambda^n(S)\le \frac{1}{e}\right\}.
\end{equation*}
where we use the convention $\inf\emptyset =+\infty$.  For fixed $\lambda>0$ and $n\ge 1$,  by using  Properties \ref{item:1} and \ref{item:2} and continuity, we have
\begin{align}
  \label{eq:1}
  &\forall \ell\ge \mathrm L_\lambda(n) && \Prb_{\lambda}^n(\mathrm B_n\longleftrightarrow\partial \mathrm B_\ell)\le \exp\left(-\left\lfloor{\ell}/{\mathrm L_\lambda(n)}\right\rfloor\right)\\
  &\forall \lambda'>\lambda&& \Prb_{\lambda'}^n(\mathrm B _n \longleftrightarrow\partial \mathrm B_{L_\lambda(n)})\ge \frac{\lambda'-\lambda}{e\lambda'}.\label{eq:2}
\end{align}

In the subcritical  regime and up to logarithmic corrections, the quantity $\mathrm L_\lambda(n)$  can be related to the more standards definition of correlation length for the truncated model (defined as the exponential rate of decay of the the connection probabilities).


\subsection{Behaviour of the correlation length in the fidicious regime}
By definition, the correlation length satisfies $\mathrm L_\lambda(n) \in [n,\infty]$. In order to prove  Proposition \ref{prop:seed}, we will use that the correlation length is not degenerated in the fidicious regime, which is the content of the following lemma.
\begin{lem}\label{lem:Lbound}
For any $\lambda\in (\widehat\lambda_c(\delta),\lambda_c(\delta))$, there exists $n_0\ge 1$ such that 
\begin{equation}
    \forall n \ge n_0\qquad 3n \le \mathrm{L}_\lambda(n)<+\infty.
\end{equation}
\end{lem}
\begin{proof}
  Let $\lambda>0$.   We prove the following two implications, from which the lemma follows.
  \begin{align*}
    \left(\exists n\ge1\qquad     \mathrm L_\lambda(n)=+\infty\right)\implies  \lambda  \ge \lambda_c(\delta),  \\
    \left( \mathrm L_\lambda(n)\le 3n \text{ for infinitely many }n\ge1 \right)\implies  \lambda  \le \widehat\lambda_c(\delta).
  \end{align*}
  The first implication is a direct consequence of Equation~\eqref{eq:2}. For the second implication, let $n$ such that $\mathrm L_\lambda(n)\le 3n$ and set $N:= n^{1+\frac{\delta}{2d}}$.  If $\mathrm B_N$ is connected to $\partial \mathrm B_{2N}$, either the connection uses only balls of radius smaller than or equal to $n$, or there must exist a ball of radius at least $n$ intersecting $\mathrm B_{2N}$. Therefore, by Markov inequality 
  \begin{equation}
    \label{eq:7}
       \Prb_{\lambda,\mu_\delta}(\mathrm B_{N}\leftrightarrow\partial \mathrm B_{2N})\le \Prb_{\lambda}^n(\mathrm B_{N}\leftrightarrow\partial \mathrm B_{2N})+\mathbb E_{\lambda,\mu}(|\eta\cap E| ),
  \end{equation}
  where $E:=\left\{(x,r)\in \sR^d\times [n,+\infty[: \|x\|_2\le 2N+r\right\}$.  Consider a covering of $\rB_{N}$ with at most $C (N/n)^d$ balls of radius $n$ with center on $\rB_N$ (where $C$ is a constant depending only on $d$). If $\rB_N$ is connected to $\partial \rB_{2N}$, then at least one of the balls of the covering must be connected to $\partial \rB_{2N}$. By the union bound,  translation invariance,  we obtain
  \begin{align}
    \label{eq:8}
     \Prb_{\lambda}^n(\mathrm B_{N}\leftrightarrow\partial \mathrm B_{2N})\le C\left(\tfrac N n\right)^d \mathbb P_\lambda^n(\rB_n\longleftrightarrow \partial \rB_N)\overset {\eqref{eq:1}}\le C\left(\tfrac N n\right)^d \exp\big(-\lfloor N/3n\rfloor\big),
  \end{align}
  where we used the hypothesis $\mathrm L_\lambda(n)\le 3n$ on the second inequality.

  Let us now bound the second term in \eqref{eq:7}. The quantity $|\eta\cap E|$ is a Poisson random variable with parameter  $(dz\otimes\mu_\delta)(E)$, hence
\begin{equation}\label{eq:interbigball}
\mathbb E_{\lambda,\mu}(|\eta\cap E| )=(dz\otimes\mu_\delta)(E)\le C' \int_{n}^{\infty}(N+t)^dd\mu_\delta(t)\le C''\frac{N^ d}{n^ {d+\delta}}+C''\int_N^ \infty t^ dd\mu_\delta(t),
\end{equation}
  where $C',C''$ are constants depending only on $d$. 

  Combining the previous inequalities that hold for an infinite number of $n$, we get that
  \begin{equation*}
    \inf_{N\ge 1}\Prb_{\lambda,\delta}(\mathrm B_{N}\longleftrightarrow\partial \mathrm B_{N})=0,
  \end{equation*}
  which concludes that $\lambda\le \widehat \lambda_c(\delta)$.
\end{proof}

\begin{lem}\label{lem:N}
Let $\ep>0$, $\lambda_0\in(\widehat\lambda_c(\delta),\lambda_c(\delta))$.  For any $n_0\ge 1$, there exist $n\ge n_0$ and $\lambda\in\left(\frac{\widehat\lambda_c(\delta)+\lambda_0}{2},\lambda_0\right)$ such that 
\begin{equation}
    \mathrm{L}_{\lambda-\frac{1}{\log n}}(n^ {1+\ep})\ge \mathrm{L}_\lambda(n)^ {1+\frac{\ep}{2}}.
\end{equation}
\end{lem}

\begin{proof}
For contradiction, let $n_0\ge 1$, and assume that for every $n\ge n_0$ and $\lambda \in\left(\frac{\widehat\lambda_c(\delta)+\lambda_0}{2},\lambda_0\right)$
  \begin{equation}
    \label{eq:4}
       \mathrm{L}_{\lambda-\frac{1}{\log n}}(n^ {1+\ep})\le \mathrm{L}_\lambda(n)^ {1+\frac{\ep}{2}}.
     \end{equation}
     Let $N_0\ge n_0 $, and set   $N_k:=N_0^{(1+\varepsilon)^k}$ and $\lambda_k:=   \lambda_0-\sum_{0\le i\le k-1}\frac1{\log(N_i)}$  for every $k\ge 1$.  By choosing $N_0$ large enough we can  ensure that $\lambda_k>\frac{\widehat\lambda_c(\delta)+\lambda_0}{2}$ for every $k$. Thanks to Lemma \ref{lem:Lbound}, we have that $L_{\lambda_0}(N_0)<+\infty$. By induction, Inequality~\eqref{eq:4} implies     \begin{equation}
       \label{eq:6}
       \forall k\ge 0\qquad\mathrm  L_{\lambda_k}(N_k)\le\mathrm L_{\lambda_0}(N_0)^{(1+\frac{\ep}{2})^k},
     \end{equation}
     which contradicts the lower bound $L_{\lambda}(n)\ge n$.
     
\end{proof}


\begin{proof}[Proof of Proposition \ref{prop:seed}]
Let $\delta>0$. Assume that $\widehat\lambda_c(\delta)<\lambda_c(\delta)$. Let $\ep>0$, $\lambda_0\in(\widehat\lambda_c(\delta),\lambda_c(\delta)) $. By Lemma \ref{lem:Lbound}, there exists $n_0\ge 1$ such that 
\begin{equation}\label{eq:defn0}
   \forall n \ge n_0\quad   \forall \lambda\ge \frac{\widehat\lambda_c(\delta)+\lambda_0}{2}\qquad \mathrm{L}_\lambda(n)\ge \mathrm{L}_{\frac{\widehat\lambda_c(\delta)+\lambda_0}{2}}(n)\ge3n .
\end{equation}
By Lemma \ref{lem:N}, there exist $n\ge n_0$ and  $\lambda\in(\frac{\widehat\lambda_c(\delta)+\lambda_0}{2},\lambda_0)$ such that 
\begin{equation}\label{eq:ineqN}
    \mathrm{L}_{\lambda-\frac{1}{\log n}}(n^ {1+\ep})\ge \mathrm{L}_\lambda(n)^ {1+\frac{\ep}{2}}.
\end{equation}
Set 
\[N:=\frac1{9d}\cdot { \mathrm{L}_{\lambda-\frac{1}{\log n}}(n^ {1+\ep})}.\]
Thanks to inequalities \eqref{eq:1} and \eqref{eq:2}, we have 
\begin{equation}\label{eq:expdecay}
 \forall \ell\ge \rL_\lambda(n)\qquad\Prb_\lambda^n(\mathrm B_{n}\longleftrightarrow\partial \mathrm B_{\ell} )\le \exp\left(-\left\lfloor{\ell}/{\rL_\lambda(n)}\right\rfloor\right)
 \end{equation}
 and
 \begin{equation}\label{eq:eqNub}
 \Prb^{n^{1+\ep}}_\lambda(\mathrm B_{n^{1+\ep}}\longleftrightarrow\partial\mathrm B_{9dN})\ge \frac{1}{\log n}.
 \end{equation}
 By the union bound, there exists a constant $c_d$ depending only on $d$ such that
  \begin{equation*}
   \Prb^{n^{1+\ep}}_\lambda(\mathrm B_{n^{1+\ep}}\longleftrightarrow\partial\mathrm B_{9dN})\le c_d n^{(d-1)\ep} \,\Prb_{\lambda}^{n^{1+\ep}}(\mathrm B_{n}\longleftrightarrow\partial\mathrm B_{3dN}) 
  \end{equation*}
  where we use that by inequality \eqref{eq:defn0}, we have $\rL_{\lambda-1/\log n}(n^ {1+\ep})=9dN\ge 3n^{1+\ep}$. 
  Hence, we obtain using inequality \eqref{eq:eqNub}
 \begin{equation}\label{eq:nottoobadconnexion}
 \Prb_{\lambda}^{n^{1+\ep}}(\mathrm B_{n}\longleftrightarrow\partial\mathrm B_{3dN}) \ge \frac{1}{c_d n^{(d-1)\ep}\log n }\ge\frac{3}{n^ {\ep'}}
 \end{equation}
 where $\ep'=d\ep$ (up to choosing a larger $n_0$ depending on $\ep$, we can assume that $n$ is large enough such that the last inequality holds).
 Let $\rho \in[1, 2d]$.
By union bound, we have
\begin{equation}\label{eq:ubnottoobad}
    \begin{split}
        \Prb_{\lambda}^ n( \mathrm B_{(\rho+1/4) N} \longleftrightarrow \partial \mathrm B_{(\rho+1/2) N})&\le c_d \left(\frac{\rho N}{n}\right)^ {d-1}\Prb_{\lambda}^ n( \mathrm B_n \longleftrightarrow \partial  \mathrm B_{N/2})\\&\le   c_d \left(\frac{2d N}{n}\right)^ {d-1} \exp\left(-\frac{N^ {\frac{\ep}{2+\ep}}}{(9d)^ 2}\right)
    \end{split}
\end{equation}
where we used inequalities \eqref{eq:ineqN} and \eqref{eq:expdecay} in the last inequality.
By Lemma \ref{lem:Lbound}, we have $3dN\ge n^ {1+\ep}$, up to choosing a larger $n_0$ depending on $\ep$, we can assume that $n$ is large enough such that  
\[ c_d \left(\frac{2d N}{n}\right)^ {d-1} \exp\left(-\frac{N^ {\frac{\ep}{2+\ep}}}{(9d)^ 2}\right)\le\frac{1}{n^ {\ep'}}\]
 Combining inequalities \eqref{eq:nottoobadconnexion} and \eqref{eq:ubnottoobad}, it follows that for any $\rho\in[1,2d]$
\begin{equation*}
\begin{split}
        \Prb_{\lambda,\mu_\delta}( \cE_{n,N}( \rho))&\ge\Prb_{\lambda}^{n^{1+\ep}}(\mathrm B_{n}\longleftrightarrow\partial\mathrm B_{3dN})- \Prb_{\lambda}^n( \mathrm B_{(\rho+1/4) N} \longleftrightarrow \partial \mathrm B_{(\rho+1/2) N})\\&-\Prb_{\lambda,\mu_\delta}(\exists (x,r)\in \mathrm B_{\rho N} \times\sR_+: \mathrm B_r^ x\cap \partial  \mathrm B_{(\rho+1/4) N})\\&\ge \frac{2}{n^ {\ep'}}-\frac{c_d}{n^ \delta}\ge \frac{1}{n^ {\ep'}}.
        \end{split}
\end{equation*}
The result follows from the fact that $ \Prb_{\lambda,\mu_\delta}( \cE_{n,N}( \rho))\le  \Prb_{\lambda_0,\mu_\delta}( \cE_{n,N}( \rho))$.

\end{proof}

\section{Sharp threshold}\label{sec:ST}
In this section, we prove a Talagrand formula for inhomogeneous weights. By a proper encoding of a Poisson point process, we prove a Talagrand formula for Boolean percolation. We apply this result to prove Proposition \ref{prop:ST}.
\subsection[Why do we need to sprinkle in delta]{Why do we need to sprinkle in $\delta$?}\label{sec:example}
A sprinkling in $\lambda$ is not sufficient to prove a sharp threshold in full generality. Let us illustrate this fact by a simple example. Denote by $\cF$ the event that there exists an open ball of radius at least $n^ {d/(d+\delta)}$ centered inside the ball $\mathrm B_n$. We have
for every $\lambda>0$, \[\Prb_{\lambda,\mu_\delta}(\cF)= 1-\exp(-\alpha_d\lambda),\] where $\alpha_d$ is the volume of the unit ball in $\mathbb R^d$. Hence, even if the event $\cF$ has a lot of symmetries (the event is invariant under rotations), there is no sharp threshold in $\lambda$. The big balls, that are present with a very small probability, play a dictator role in the sense that the presence of one ball is enough for the event to occur. In the event $\cE_{n,N}(\rho)$ (defined in Section \ref{sec:main}), the big balls may also play a dictator role. As for the event $\cF$, we cannot hope to prove a sharp threshold in $\lambda$ for the event $\cE_{n,N}(\rho)$.
 By slightly reducing $\delta$, the density of big balls increase faster compare to the density of small balls. Going back to the event $\cF$, we have for $\ep>0$ \[\Prb_{\lambda,\mu_{\delta-\ep}}(\cF)=1-\exp(-\lambda n^ {d\ep/(d+\delta)}).\] Hence, the event $\cF$ has a sharp threshold in $\delta$.

\subsection{Sharp threshold for inhomogeneous percolation}
In this section, we state general results on sharp threshold for inhomogeneous weights.
We will need the following result that is a Talagrand formula that holds for inhomogeneous Bernoulli percolation.
Let $N\ge1$, and $p_1,\dots,p_N\in[0,1]$, set
\[\mathrm P_{p_1,\dots,p_N}:=\otimes_{i=1}^N \mathrm {Ber}(p_i)\]
where $\mathrm{Ber}(p)$ is the distribution on $\{0,1\}$ of Bernoulli of parameter $p$.
For $i\in\{1,\dots,N\}$ define $\tau_i$ the function on $\{0,1\}^{N}$ that switches the value of the $i$-th bit.
\begin{prop}[Inhomogeneous Talagrand]\label{lem:inhomogeneous Talagrand}There exists a universal constant $C>0$ such that for any family $(p_i)_{1\le i \le N}\in(0,1/2]^N$ for any $f:\{0,1\}^N \rightarrow\{0,1\}$.
We have 
\[\sum_{i=1}^N p_i |\log p_i|\Inf_i(f)\ge C\Var(f)\log \left(\frac{1}{\max _{1\le i \le N}p_i\Inf_i(f)}\right)\,\]
where
\[\Inf_i(f):=\mathrm P_{p_1,\dots,p_N}(f\circ\tau_{i}\ne f).\]
\end{prop}
\begin{rk}
For simplicity, we prove a version of inhomogeneous Talagrand for $p$'s that are close to $0$. We did not try to optimize the proof to obtain a more symmetric version.
\end{rk}
We were not able to find a proof of this inhomogeneous version of Talagrand formula in the literature. For completeness, we include a proof of Proposition \ref{lem:inhomogeneous Talagrand} in the appendix.

\subsection{Inhomogeneous Talagrand for Boolean percolation}
In this section, we prove a Talagrand formula for Boolean percolation. 
For $x\in\sZ^d$ and $n\ge 1$, define
\[\rS_n^x:= \left(x+\left[-\frac{1}{2},\frac{1}{2}\right)^d\right)\times [n,n+1)\,.\]
Let $\delta>0$ and $\lambda>0$. For short, we will write $\Prb_{\delta}$ for $\Prb_{\lambda,\mu_{\delta}}$. The parameter $\lambda$ will remain fixed in all this section.

For $\cE$ a finite increasing event, denote
\[\mathrm{Piv}_n^x(\cE):=\int_{(z,r)\in\rS_n^x}\Prb_{\delta}(\eta \notin \cE, \eta\cup\{(z,r)\}\in \cE)dzd\mu_\delta(r).\]
Set $\Omega$ the set of countable subsets of $\sR^ d\times\sR_+$. We define the following distance between points in $\Omega$: for any $\eta,\eta'\in\Omega$
\[ \dis(\eta,\eta'):=\left\{\begin{array}{ll}+\infty &\mbox{ if $|\eta|\ne|\eta'|$}\\
\inf_{\pi:\eta\mapsto \eta' \text{ bijection}}\sup\{\|x-x'\|+|r-r'|: (x',r')=\pi((x,r))\}&\mbox{otherwise}\end{array}\right.\]
where we use the convention that $|\eta|=|\eta'|$ if both sets are infinite.
We will denote by $\partial \cE$ the boundary of $\cE$ for the topology associated with the distance $\dis$ in the space $\Omega$. 
For $\eta\in\Omega$ and $N\ge1$, we define the neighborhood $\mathcal N_N(\eta,\ep)$ at distance $\ep$ restricted to $\rB_N$ 
\[\mathcal N_N(\eta,\ep):=\left\{\eta'\in\Omega:\,\dis(\eta\cap (\mathrm B_N\times\sR_+),\eta'\cap (\mathrm B_N\times\sR_+))\le \ep\right\}\]
Let $N$ be such that the event $\cE$ only depends on balls centered in $\mathrm B_N$.
We will need the following notion of approximate pivotality at level $\ep>0$
\[\mathrm{Piv}_n^{x,\ep}(\cE):=\int_{(z,r)\in\rS_n^x}\Prb_{\delta}(\exists \eta'\in\mathcal N_N(\eta,\ep):\eta' \notin \cE,\eta'\cup\{(z,r)\}\in \cE)dzd\mu_\delta(r).\]

The following proposition is a Talagrand formula for increasing event in Boolean percolation.
\begin{prop}\label{prop:talpoisson} Let $\lambda>0$, $\delta_0,\delta_1>0$. There exists $C_0>0$ such that the following holds. For any $\delta\in(\delta_0,\delta_1)$, For any finite increasing event $\cE$ such that $\Prb_{\delta_0}(\partial\cE)=0$, we have \[\sum_{x\in \mathbb Z^d, n\ge 1}\log n\,\mathrm{Piv}_{n}^ x(\cE) \ge C_0\Prb_{\delta}(\cE)(1-\Prb_{\delta}(\cE))\log \left(\frac{1}{\max_{x\in\sZ^d, n\ge 1}\mathrm{Piv}_{n}^ x(\cE)}\right).\]
\end{prop}
\begin{rk}
This proposition holds for increasing events for the two following partial orders. An event $\cE$ is coloring increasing if for any $\eta\in\cE$ for any $\eta'$ such that $\cO(\eta)\subset \cO(\eta')$ then $\eta'\in\cE$.

We also define the partial order $\prec$ on $\sR^d\times \sR_+$ as follows. We say that $\eta\preceq \eta'$ if for any $(x,r)\in \eta$ there exists $(y,r')\in\eta'$ such that $\mathrm B_r^x \subset \mathrm B _ {r'}^y$.
We say that an event $\cE$ is increasing for $\prec$ if  $\eta\in\cE$ and $\eta\preceq \eta'$ then $\eta'\in\cE$.

\end{rk}
\begin{rk} We believe that the hypothesis $\Prb_{\delta}(\partial\cE)=0$ holds for any increasing event. Note that we will only work on connection events. For connection events, if $\eta\in\partial \cE$ then there exists $(z,r),(z',r')\in\eta $ such that $\|z-z'\|=r+r'$. It follows that $\Prb_{\delta}(\partial\cE)=0$.

\end{rk}

\begin{rk} At a first look, it may seem that  the formula in Proposition \ref{prop:talpoisson}  involves an underlying discretization of the space. We would like to emphasize that the pivotal quantities  $\mathrm{Piv}_{n}^ x(\cE)$ are defined in terms of the continuous model itself and not a discretized version of it.  One can check that the formula is equivalent to the existence of a constant $C>0$ such that
\[\int_{(z,r)\in\sR^d \times \sR_+}I_{z,r}(\cE)\log r \,dzd\mu(r)\ge C \Prb(\cE)(1-\Prb(\cE))\log \left(\frac{1}{\sup_{(z,r)\in\sR^d \times \sR_+} I_{z,r}(\cE)}\right)\]
where $I_{z,r}(\cE)=\Prb(\eta \notin \cE, \eta\cup\{(z,r)\}\in\cE)$.

\end{rk}
We will need the following lemma to prove Proposition \ref{prop:talpoisson}.
\begin{lem}\label{lem:approxpiv}Let $\cE$ be a finite increasing event such that for any $\lambda,\delta>0$, $\Prb_{\delta}(\partial\cE)=0$. We have for any $x\in\sZ^ d$ and $n\ge 1$
\[\lim_{\ep\rightarrow 0}\mathrm{Piv}_n^{x,\ep}(\cE)=\mathrm{Piv}_n^x(\cE).\]
\end{lem}

\begin{proof} Let $N$ be such that $\cE$ only depends on points centered in $ \mathrm B_N\times \sR_+$. 
Let $\ep>0$ and $(z,r)\in\rS_n^x$, we have
\begin{equation*}
    \begin{split}
        \Prb_{\delta}&(\exists \eta'\in\mathcal N_N(\eta,\ep):\eta' \notin \cE,\eta'\cup\{(z,r)\}\in \cE)\\&\le   \Prb_{\delta}( \eta \in \mathcal N_N(\partial\cE,\ep) )+
        \Prb_{\delta}( \eta\cup\{(z,r)\} \in \mathcal N_N(\partial\cE,\ep))+ \Prb_{\delta}(\eta \notin \cE,\eta\cup\{(z,r)\}\in \cE).
    \end{split}
\end{equation*}
It follows that 
\begin{equation*}
\begin{split}
   \mathrm{Piv}_n^{x,\ep}(\cE)\le&\mathrm{Piv}_n^x(\cE)+\Prb_{\delta}( \mathcal N_N(\partial\cE,\ep))+\int_{(z,r)\in\rS_n^x}\Prb_{\delta}( \eta\cup\{(z,r)\}\in  \mathcal N_N(\partial\cE,\ep) )dzd\mu_\delta(r).
    \end{split}
\end{equation*}
We have
\[\lim_{\ep\rightarrow 0}\Prb_{\delta}(\mathcal N_N(\partial\cE,\ep))=\Prb_{\delta}(\partial \cE)=0.\]
Similarly, by dominated convergence
\[\lim_{\ep\rightarrow 0}\int_{(z,r)\in\rS_n^x}\Prb_{\delta}( \eta\cup\{(z,r)\}\in \mathcal N_N(\partial\cE,\ep))dzd\mu_\delta(r)=\int_{(z,r)\in\rS_n^x}\Prb_{\delta}(\eta\cup\{(z,r)\}\in\partial \cE)dzd\mu_\delta(r).\]
Using Proposition 2.1 in \cite{Last14}, we have
\begin{equation*}
    \begin{split}
\int_{(z,r)\in\rS_n^x}\Prb_{\delta}(\eta\cup\{(z,r)\}\in\partial \cE)dzd\mu_\delta(r)&\le\int_{(z,r)\in\sR^d\times\sR_+}\Prb_{\delta}(\eta\cup\{(z,r)\}\in\partial \cE,\eta\notin \cE)dzd\mu_\delta(r)\\&= \frac{\partial}{\partial \lambda}\Prb_\delta(\partial \cE)=0
\end{split}
\end{equation*}
The result follows.

\end{proof}
Denote by 
\[F(t):=\int_{1}^{t}\frac{d+\delta}{r^{d+1+\delta}}dr\,.\]
For $n\ge 1$ define
\[\forall t\in[0,1]\qquad F_n^{-1}(t):=
 \inf\left\{s \ge n: \frac{F(s)-F(n)}{F(n+1)-F(n)}\ge t\right\}\,.\]
We will frequently use the following claim in the proof of Proposition \ref{prop:talpoisson}. The proof follows from an easy computation.
\begin{claim}\label{claim:reglaw}There exists a constant $c$ depending on $\delta$ and $d$ such that the following holds. Let $n\ge 1$, $i\ge1$. Let $t_0<t_1\in[n,n+1)$.
We have
\[F_n^{-1}(t_1)-F_n^{-1}(t_0)\le c(t_1-t_0),\]
Moreover, if
\[F_n^{-1}(t_1)-F_n^{-1}(t_0)\le\frac{1}{2^i},\]
then
\[t_1-t_0\le \frac{1}{2^i}.\]
\end{claim}

\begin{proof}[Proof of Proposition \ref{prop:talpoisson}]
Fix $\lambda$ and $\delta_0<\delta_1$.
Set 
\[g(n):=\frac{F(n+1)-F(n)}{d+\delta}\,\]
and
\[ p(k,t):=\Prb(\cP(t)\ge k+1|\cP(t)\ge k)\]
where $\cP(t)$ is a Poisson random variable of parameter $t$.
\paragraph{Encoding the Poisson point process.}
For each $\rS_n^x$, we associate a Bernoulli random variable $\alpha_0$ of parameter $p(0,\lambda g(n))$. This random variable will encode the presence of at least one point in the cell $\rS_n^ x$.
For each $(\rS_n^x,k)$ with $k\ge 1$, we associate a Bernoulli random variable $\alpha_k$ of parameter $p(k,\lambda g(n))$ (that will encode the presence of a $k$-th ball in the cell $\rS_n^x$), an i.i.d. family $(\beta_i)_{i\ge 1}$ of Bernoulli random variables of parameter $1/2$ and an i.i.d. family $(\gamma_i^{(1)},\dots,\gamma_i^{(d)})_{i\ge 1}$ of Bernoulli random variables of parameter $1/2$ (that will encode respectively the radius and the position of the $k$-th ball of the cell). Write $\mathbb P$ for the joint law of these families of Bernoulli random variable.
Finally, define for $m\in\sN\cup\{+\infty\}$
\[z_k^m(\rS_n^ x)=z_k^m:=x+ \left(\sum_{i= 1}^m \frac{\gamma_i^{(1)}}{2^i},\dots, \sum_{i= 1}^m\frac{\gamma_i^{(d)}}{2^i}\right)-\left(\frac{1}{2},\dots,\frac{1}{2}\right)\]
and 
\[r_k^m(\rS_n^ x)=r_k^m:= F_n^{-1}\left(\sum_{i= 1}^m\frac{\beta_i}{2^i}\right)\,.\]
When $m=+\infty$, we write $z_k$ and $r_k$ instead of  $z_k^ {+\infty}$ and $r_k^ {+\infty}$.
It is easy to check that $z_k$ is distributed uniformly on $x+[-1/2,1/2)^d$ and $r_k$ is distributed given the distribution $(d+\delta)\mu_\delta$ conditioned on being in the interval $[n,n+1]$. 
Denote by $\eta_{\rS_n^ x}$ the following random set:
\[\eta_{\rS_n^ x}:=\left\{(z_k,r_k): \prod_{i=0}^{k-1}\alpha_i=1\right\}.\]
We have
\[\Prb(|\eta_{\rS_n^ x}|=k)=\Prb\left(\alpha_k=0, \prod_{i=0}^{k-1}\alpha_i=1\right)=(1-p(k,\lambda g(n)))\prod_{i=0}^{k-1}p(i,\lambda g(n))=\Prb( \cP(\lambda g(n))=k).\]
Hence, the set $\eta_{\rS_n^ x}$ is a Poisson point process of intensity $\lambda dz\otimes \mu_\delta$ restricted to the cell $\rS_n^x$. It yields that the set $\eta=\cup_{x,n}\eta_{\rS_n^ x}$ is a Poisson point process of intensity $\lambda dz\otimes \mu_\delta$. In particular, we have
\begin{equation}\label{eq:equalitylaw}
    \Prb(\eta\in\cE)=\Prb_{\delta}(\cE).
\end{equation}


\paragraph{Approximation of the Poisson point process using a finite encoding.} In order to apply Talagrand formula, we need to approximate the Poisson point process with a finite encoding. First, since the event $\cE $ is finite, there exists $N$ such that the event only depends on balls in $\cup_{x\in\mathrm B_{N}\cap \sZ^d,n\ge 1}\rS_n^x$.
Let $K\ge 1$. Denote by $\cF_K$ the event where there exists a cell with more than $K$ balls or a ball whose radius is larger than $KN$:
\[\cF_K:=\left\{\exists(x,n)\in(\mathrm B_N\cap\sZ^d)\times [0,KN]: \,|\eta_{\rS_n^ x}|\ge K \right\}\cup\left\{\eta\cap (\mathrm B_N\times[KN,+\infty))\ne\emptyset \right\} .\]
It is easy to check that
\[\lim_{K\rightarrow\infty}\Prb(\cF_K)=0\,.\]
Write $\overline z_k$ and $\overline r_k$ for $z_k^K$ and $r_k^K$. 
We now define $\overline \eta$ the approximate version of $\eta$ defined by truncating up to level $K$:
\[\overline \eta_{\rS_n^ x}:=\left\{(\overline z_k,\overline r_k): \prod_{i=0}^{k-1}\alpha_i=1, \,k\le K\right\}\]
and \[\overline \eta:=\bigcup_{(x,n)\in(\mathrm B_N\cap\sZ^d)\times [0,KN]} \overline \eta_{\rS_n^ x}.\]
Note that on the event $\cF_K^c$, we have
\begin{equation}\label{eq:diseta}
   \overline \eta \in\mathcal N_N(\eta, \ep_K)
\end{equation}
where $\ep_K$ depends on $K$ and goes to $0$ when $K$ goes to infinity.
Fix $(x,n)\in(\mathrm B_N\cap\sZ^d)\times [0,KN]$. 
For $(z,r)\in \rS_n^x$, we define its projection $(\pi(z),\pi_n(r))$ as follows
\[(\pi(z),\pi_n(r)):=\left(2^{-K}\lfloor 2^Kz\rfloor, F_n^{-1}\left(2^{-K}\left\lfloor 2^K\frac{F(r)-F(n)}{F(n+1)-F(n)}\right\rfloor\right)\right)\,.\]
In particular we have 
\[(\overline z_k,\overline r_k)=(\pi(z_k),\pi_n(r_k))\,.\]
Thanks to Claim \ref{claim:reglaw}, we have for $r\in[n,n+1)$
\begin{equation}\label{eq:comprpir}
    0\le r-\pi_n(r)\le c2^{-K}.
\end{equation}
We need to compute the influence of all the variables on the event $\ind_{\overline \eta\in\cE}$.
We define the influence as follows 
\[\Inf_{\alpha_0}^ {(x,n)}(\cE):=\Prb(\ind_{\overline \eta \in\cE}\circ\tau_{\alpha_0}^{(x,n)}\ne \ind_{\overline \eta \in\cE})\]
where $\tau_{\alpha_0}^{(x,n)}$ is the function that switches the value of the bit corresponding to the variable $\alpha_0$ of the cell $\rS_n^x$. The other influences are defined similarly.

\paragraph{Influence of $\alpha_0$.} The variable $\alpha_0$ dictates if there is at least one point in $\rS_n^x$. If the variable is influential then the cell $\rS_n^x$ is pivotal. We say that the cell is pivotal if $\overline \eta\setminus \rS_n^x \notin\cE$ and $\overline \eta\setminus \rS_n^x\cup\{(x,n+1+\sqrt d)\}\in\cE$.
It yields that
\begin{equation*}
\begin{split}
\Inf_{\alpha_0}^ {(x,n)}(\cE)&\le \Prb( \overline \eta \setminus \rS_n^x\notin \cE, \overline \eta\setminus \rS_n^x\cup\{(x,n+1+\sqrt d)\}\in \cE)\\&=\frac{\Prb( \overline \eta \notin \cE, \overline \eta\cup\{(x,n+1+\sqrt d)\}\in \cE)}{\Prb(\overline \eta\cap\rS_n^x=\emptyset)} .
\end{split}
\end{equation*}
We have
\begin{equation*}
    \begin{split}
        \Prb&( \overline \eta \notin \cE, \overline \eta\cup\{(x,n+1+\sqrt d)\}\in \cE)\\
        &= \frac{1}{g(n+2d)}\int_{(z,r)\in\rS_{n+2d}^x} \Prb( \overline \eta \notin \cE, \overline \eta\cup\{(x,n+1+\sqrt d)\}\in \cE)dzd\mu_\delta(r)\\
        &\le\frac{1}{g(n+2d)}\int_{(z,r)\in\rS_{n+2d}^x} \Prb( \overline \eta \notin \cE, \overline \eta\cup\{(z,r)\}\in \cE)dzd\mu_\delta(r)
    \end{split}
\end{equation*}
and using \eqref{eq:diseta}
\begin{equation*}
  \frac{1}{g(n+2d)}\int_{(z,r)\in\rS_{n+2d}^x} \Prb( \overline \eta \notin \cE, \overline \eta\cup\{(z,r)\}\in \cE)dzd\mu_\delta(r)\le \Prb(\cF_K)+\frac{\mathrm{Piv}_{n+2d}^{x,\ep_K}(\cE)}{g(n+2d)} .
\end{equation*}
Besides, we have
\[\Prb(\overline \eta\cap\rS_n^x=\emptyset)=\exp(-\lambda g(n))\ge \exp(-\lambda)\,.\]
Finally, we have
\[\Inf_{\alpha_0}^ {(x,n)}(\cE)\le \left(\Prb(\cF_K)+\frac{\mathrm{Piv}_{n+2d}^{x,\ep_K}(\cE)}{g(n+2d)}\right)\exp(\lambda).\]
\paragraph{Influence of $\alpha_k$ for $k\le K-1$.}
For $\alpha_k$ to be influential, we must have that 
$\prod_{i=0}^{k-1}\alpha_i=1$. This occurs with probability 
$\Prb( \cP(\lambda g(n))\ge k)$. Moreover, the cell $\rS_n^x$ must be pivotal.
By the same computations as above, we have
\[\Inf_{\alpha_k}^ {(x,n)} (\cE)\le \Prb( \cP(\lambda g(n))\ge k) \left(\Prb(\cF_K)+\frac{\mathrm{Piv}_{n+2d}^{x,\ep_K}(\cE)}{g(n+2d)}\right)\exp(\lambda).\]

Fix $1\le k\le K$. We now consider the influence of variables encoding the position and the radius of the $k$-th ball of the cell $\rS_n^ x$.
\paragraph{Influence of $\gamma_j^{(1)}$ for $j\le K$.} For $\gamma_j^{(1)}$ to be influential, the $k$-th ball of the cell $\rS_n^ x$ has to be present, that is $\prod_{i=0}^{k-1}\alpha_i=1$. This occurs with probability 
$\Prb( \cP(\lambda g(n))\ge k)$.

Denote by $\Lambda$ the random box of side length $2^{-j+1}$ centered at a point in $2^{-j+1}\sZ^d$ that contains $\overline z_k$. Set $\overline \eta^{(k)}=\overline \eta\setminus \{(\overline z_k,\overline r_k)\}$. For $\gamma_j^{(1)}$ to be pivotal the function $z\mapsto \ind_{\overline \eta^{(k)}\cup\{(z,\overline r_k)\}\in \cE}$ must be non-constant on $\Lambda$.
Let us assume that the function is non-constant on $\Lambda$.
Set 
\[a:=\inf\{t\ge n :\exists z\in\Lambda\quad\overline\eta^{(k)}\cup\{(z,t)\}\in \cE\}\]
and
\[b:=\sup\{t\le n+1 : \exists z\in\Lambda\quad\overline\eta^{(k)}\cup\{(z,t)\}\notin \cE\}.\]
It is easy to check that $\overline\eta^{(k)}\cup\{(w,a+d2^{-j+1})\}\in \cE$ for all $w\in\Lambda$. Hence, we have $a\le b\le a+d2^{-j+1}$.
If $\gamma_j^{(1)}$ is pivotal, then $\overline r_k=\pi_n(r_k)\in[a,b]$ and by inequality \eqref{eq:comprpir}, we have $\overline r_k\in [a,b+c2^{-K}]$.
Using Claim \ref{claim:reglaw}, we have
\begin{equation*}
    \begin{split}
          & \int_{(z,r)\in\rS_{n}^x}\Prb\left(\overline\eta^{(k)}\notin \cE,\overline\eta^{(k)}\cup\{(\pi(z),\pi_n(r))\}\in \cE, \pi_n(r)\in[a,b]\right)dzd\mu_\delta(r)\\
              &\quad\le \int_{(z,r)\in\rS_{n}^x}\Prb\left(\overline\eta\setminus \rS_n^ x\notin \cE,\overline\eta\setminus \rS_n^ x\cup\{(z,n+d)\}\in \cE, r\in[a,b+c2^{-K}]\right)dzd\mu_\delta(r)\\
                &\quad\le 2c^2d2^{-j+1}g(n)\int_{z\in (x+[-1/2,1/2)^d)}\Prb\left(\overline\eta\setminus \rS_n^ x\notin \cE,\overline\eta\setminus \rS_n^ x\cup\{(z,n+d)\}\in \cE\right)dzd\mu_\delta(r)\\
        &\quad\le c^2d2^{-j+2}\frac{g(n)}{g(n+d)}\int_{(z,r)\in\rS_{n+d}^x}\Prb\left(\overline\eta\setminus \rS_n^ x\notin \cE,\overline\eta\setminus \rS_n^ x\cup\{(z,n+d)\}\in \cE\right)dz\\
         &\quad\le c^2d2^{-j+2} \frac{g(n)}{g(n+d)}(\mathrm{Piv}_{n+d}^{x,\ep_K}(\cE)+g(n+d)\Prb(\cF_K))\exp(\lambda).
    \end{split}
\end{equation*}
Finally, we have
\[\Inf_{\gamma_j^ {(1)}}^ {(x,n,k)}(\cE)\le  \Prb( \cP(\lambda g(n))\ge k)c^2d2^{-j+2}\left(\frac{\mathrm{Piv}_{n+d}^{x,\ep_K}(\cE)}{g(n+d)}+\Prb(\cF_K)\right)\exp(\lambda).\]
The same computations also hold for $\gamma_j^{(l)}$ for $l\in\{2,\dots,d\}$.

\paragraph{Influence of $\beta_j$ for $j\le K$.} Again for $\beta_j$ to be influential, the $k$-th ball has to be present, that is, we have $\prod_{i=0}^{k-1}\alpha_i=1$. This event occurs with probability 
$\Prb( \cP(\lambda g(n))\ge k)$. Set $\overline\eta^{(k)}=\overline\eta\setminus \{(\overline z_k,\overline r_k)\}$.
For $\beta_j$ to be influential the function $r\mapsto \ind_{\overline\eta^{(k)} \cup\{(\overline z_k,r)\}\in\cE}$ must be non-constant on $[n,n+1]$.
Set for $z\in (x+[-1/2,1/2)^ d)$
\[\mathfrak r(z):=\inf\{t\ge n :\overline\eta^{(k)} \cup\{(z,t)\}\in\cE \} .\]
By Claim \ref{claim:reglaw}, it is easy to check that if $\overline r_k\notin [\mathfrak r(\overline z_k)-2^ {-j},\mathfrak r(\overline z_k)+2^ {-j}]$ then $\beta_j$ is not pivotal.
By similar computations as in the previous case, we have
\begin{equation*}
    \begin{split}
    \int_{(z,r)\in\rS_{n}^x}\Prb&\left(\begin{array}{c}\overline\eta^{(k)}\notin \cE,\overline\eta^{(k)}\cup\{(\pi(z),\pi_n(r))\}\in \cE,\\ \pi_n(r)\in[\mathfrak r(\pi(z))-2^ {-j},\mathfrak r(\pi(z))+2^ {-j}]\end{array}\right)dzd\mu_\delta(r)\\
      &\le \int_{(z,r)\in\rS_{n}^x}\Prb\left(\begin{array}{c}\overline\eta\setminus \rS_n^ x\notin \cE,\overline\eta\setminus \rS_n^ x\cup\{(z,n+d)\}\in \cE,\\ r\in[\mathfrak r(\pi(z))-2^ {-j},\mathfrak r(\pi(z))+2^ {-j}+c2^{-K}]\end{array}\right)dzd\mu_\delta(r)\\
         &\le c^22^{-j+2} \frac{g(n)}{g(n+d)}(\mathrm{Piv}_{n+d}^{x,\ep_K}(\cE)+g(n+d)\Prb(\cF_K))\exp(\lambda).
    \end{split}
\end{equation*}
Finally,
\[\Inf_{\beta_j}^ {(x,n,k)}(\cE)\le \Prb( \cP(\lambda g(n))\ge k)c^22^{-j+2}\left( \frac{\mathrm{Piv}_{n+d}^{x,\ep_K}(\cE)}{g(n+d)}+\Prb(\cF_K)\right)\exp(\lambda).\]

\paragraph{Application of inhomogeneous Talagrand.}
First, it is easy to check that $p(k,g(n))$ is increasing in $k$ and
\[\forall k\ge 0\qquad |\log p(k,\lambda g(n))|\le |\log p(0,\lambda g(n))|.\]
Besides, there exists a constant $C_0$ depending only on $\delta_0,\delta_1$, $d$ and $\lambda$ such that for any $n\ge 1$ and $ \delta\in[\delta_0,\delta_1]$
\[ |\log p(0,\lambda g(n))|=-\log \left(1-\exp\left(-\frac{\lambda}{d+\delta}\left(\frac{1}{n^{d+\delta}}- \frac{1}{(n+1)^{d+\delta}}\right)\right)\right)\le C_0\log (n+1).\]
Using the equality 
\[\sum_{k\ge 1}\Prb(\cP(\lambda g(n))\ge k)=\lambda g(n),\] we get
\begin{equation*}
\begin{split}
    \sum_{(x,n)\in(\mathrm B_N\cap\sZ^d)\times[0,N]}&\Bigg(\sum_{k=0}^{K-1}\Big( p(k,\lambda g(n))|\log p(k,\lambda g(n)) |\Inf_{\alpha_k}^ {(x,n)}(\cE)\\
    &\hspace{3cm}+\frac{\log 2}{2}\sum_{j= 1}^ K \Big( \sum_{l=1}^ d\Inf_{\gamma_j^ {(l)}}^ {(x,n,k+1)}(\cE)+ \Inf_{\beta_j}^ {(x,n,k+1)}(\cE)\Big)\Big)\Bigg)\\
    &\le  \lambda e^ \lambda\sum_{(x,n)\in(\mathrm B_N\cap\sZ^d)\times[0,N]}g(n)\Bigg(C_0\log(n+1)\,\left(\frac{\mathrm{Piv}_{n+2d}^ {x,\ep_K}(\cE)}{g(n+2d)}+ \Prb(\cF_K)\right)\\
    &\hspace{3cm}+ 4c^ 2(d^2+1)\log 2\left(\frac{\mathrm{Piv}_{n+d}^ {x,\ep_K}(\cE)}{g(n+d)}+\Prb(\cF_K)\right)\Bigg).
    \end{split}
\end{equation*}
We can now apply Proposition \ref{lem:inhomogeneous Talagrand} and the previous inequality
\begin{equation}\label{eq:eqconc}
    \begin{split}
         &\lambda e^ \lambda C' \sum_{(x,n)\in(\mathrm B_N\cap\sZ^d)\times[0,N]}\log(n+1)  \,(\mathrm{Piv}_{n+2d}^ {x,\ep_K}(\cE)+ g(n)\Prb(\cF_K))\\
        &\ge C\Prb(\overline \eta \in \cE)(1-\Prb(\overline \eta \in \cE))\left|\log \left(\max_{(x,n)\in(\mathrm B_N\cap\sZ^d)\times[d,N]}2g(n)\lambda e^ \lambda c^2\left(\frac{\mathrm{Piv}_{n}^ {x,\ep_K}(\cE)}{g(n+d)}+\Prb(\cF_K)\right)\right)\right|
    \end{split}
\end{equation}
where $C'$ is a constant depending only on $\delta_0,\delta_1$, $\lambda$ and $d$.
Recall that we have the equality in law \eqref{eq:equalitylaw}.
Besides, we have
\[|\Prb(\overline \eta \in \cE)-\Prb( \eta \in \cE)|\le \Prb(\cF_K)+\Prb(\mathrm B(\partial \cE,\ep_K)).\]
Since $\Prb(\eta \in\partial \cE)=0$, we have
\[\lim_{K\rightarrow\infty}\Prb(\overline \eta \in \cE)=\Prb(\eta \in \cE)\,.\]
The result follows by letting $K$ goes to infinity in inequality \eqref{eq:eqconc} and by Lemma \ref{lem:approxpiv}.

\end{proof}

\subsection{Sharp threshold for Boolean percolation}
In this section, we prove a more general version of Proposition \ref{prop:ST} using the result of the previous section.
We say that $\cE$ is rotational invariant if for any rotation $\rho$ of center $0$, we have \[\rho(\cE):=\{ \{(\rho(x),r): (x,r)\in \eta\}: \eta\in E\}=\cE.\]
The following proposition proves a sharp threshold for finite rotational invariant events not depending on the balls centered in $\mathrm B_n$.

\begin{prop} 
Let $\lambda,\delta>0$. There exists $\kappa>0$ such that the following holds. 
For every $\varepsilon\in (0,\delta/2\kappa)$, $n\ge \kappa$  and 
$\cE\subset \sR^d\times \sR_+$ finite increasing event that does not depend on balls centered in $\mathrm B_ n$, such that $\cE$ is rotational invariant and $\Prb_{\delta_0}(\partial\cE)=0$ for any $\delta_0>0$, we have
\begin{equation*}
     \Prb_{{\delta}}(\cE)\ge \frac{1 }{n^ \ep}\implies  \Prb_{{\delta -\kappa \varepsilon}}(\cE)\ge 1- \frac{1 }{n^ \ep}.
\end{equation*}
\end{prop}

\begin{proof} Let $\lambda,\delta_1>0$. In this proof $\lambda$ is omitted from the notations as it remains fixed.
Set $\delta_0=\delta_1/2$. In what follows, we will work with $\delta\in[\delta_0,\delta_1]$. Let $\kappa$ that will be defined later. Let $n\ge \kappa$ and $\cE$ as in the statement of the proposition.

Let us first compute the derivative of $\delta \mapsto \Prb_{\delta}(\cE)$.
We have
\begin{equation*}
\begin{split}
    \mu_{\delta-t}= \frac{1}{r^{d+1+\delta-t}}\ind_{r\ge 1}dr= r^t\mu_\delta
    &=\left(1+\exp( t \log r) -1 \right)\mu_\delta\\
    &= \left(1 + t\log r + t F_r(t)\right)\mu_\delta
\end{split}
\end{equation*}
where $F_r(t):=(\exp( t \log r) -1)/t -\log r  $.
We have 
\[\lim_{t\rightarrow 0}\sup_{r\in[0,N]}|F_r(t)|=0\,.\]
We can apply Theorem 6.6 of \cite{Last14} to compute the derivative of $\delta \mapsto \Prb_{\delta}(\cE)$, we get
\begin{equation*}
    \frac{\partial}{\partial\delta}\Prb_{\delta}(\cE)=-\lambda\int_{(z,r)\in\sR^d\times \sR_+}\Prb_{\delta}(\eta \notin \cE,\eta\cup\{(z,r)\}\in \cE)\log r \,dzd\mu_\delta(r).
\end{equation*}
We now distinguish two cases.
\paragraph{First case} We assume that there exists $x\in\sZ^d$ and $m\ge 1$ such that $\mathrm{Piv}_m^x(\cE)\ge 1/\sqrt{n}$. Without loss of generality, we can assume that $m\ge 2$ since $\mathrm{Piv}_2^x(\cE)\ge c \mathrm{Piv}_1^x(\cE)$ for some constant $c$ depending on $\delta_1$. Since $\cE$ does not depend on balls centered in $\mathrm B_n$, we have $x\notin \mathrm B_{n-d}$.
It is easy to check that there exist $c_d>0$ depending only on $d$, $\rS_1,\dots ,\rS_K\subset \sR^d$ with $K\ge c_d n^{d-1}$ such that all the $\rS_i$ are disjoint and for every $i\in\{1,\dots,K\}$, there exists a rotation sending $x+[-1/2,1/2)^d$ on $\rS_i$.
By rotation invariance, it follows that
\[\int_{(z,r)\in\rS_i\times[m,m+1)}\Prb_{\delta}(\eta \notin \cE,\eta\cup\{(z,r)\}\in \cE)dzd\mu_\delta(r)\ge \frac{1}{\sqrt{n}}\]
and
 \[-\frac{\partial}{\partial\delta}\Prb_{\delta}(\cE)=\lambda\int_{(z,r)}\Prb_{\delta}(\eta \notin \cE,\eta\cup\{(z,r)\}\in \cE)\log r \,dzd\mu_\delta(r)\ge c_d \lambda \log d \,n^{d-3/2}\,.\]
 
 \paragraph{Second case} We assume that for every $x\in\sZ^d$ and $m\ge d$, we have $\mathrm{Piv}_m^x(\cE)\le 1/\sqrt{n}$.
  By Proposition \ref{prop:talpoisson}, there exists $C_0$ (depending on $\lambda$ and $\delta_1$) such that 
\begin{equation}\label{eq:approp3.4}
\sum_{x\in \mathbb Z^d, k\ge 1}\log k\,\mathrm{Piv}_{k}^ x(\cE) \ge C_0\,\Prb_{\delta}(\cE)(1-\Prb_{\delta}(\cE))\log \left(\frac{1}{\max_{x\in\sZ^d, k\ge 1}\mathrm{Piv}_{k}^ x(\cE)}\right)
\,.
\end{equation}
 We have 
 \begin{equation*}
     \sum_{x\in \mathbb Z^d, k\ge 1}\log k\,\mathrm{Piv}_{k}^ x(\cE)\le -\frac{1}{\lambda}\frac{\partial}{\partial\delta}\Prb_{\delta}(\cE).
 \end{equation*}
 By using inequality \eqref{eq:approp3.4}:
 \begin{equation}\label{eq:derivee}
    -\frac{\partial}{\partial\delta}\Prb_{\delta}(\cE)\ge \frac{2}{\kappa}\Prb_{\delta}(\cE)(1-\Prb_{\delta}(\cE))\log n
\end{equation}
 where $\kappa$ depends on $\lambda$ and $\delta_1$.
 
\paragraph{Conclusion} Finally, set  for $t\in [0, \delta_1/2\kappa] $
\[\delta(t):=\delta_1-t\kappa\]
and
\[f(t):=\Prb_{\delta(t)}(\cE).\]
We have for  $t\in[0,\delta_1/2\kappa]$, $\delta(t)\in[\delta_0,\delta_1]$.
 We have for $t\in[0,\delta_1/2\kappa]$ using inequality \eqref{eq:derivee} for the second case
\begin{equation*}
    \begin{split}
        f'(t)\ge -\kappa\frac{\partial}{\partial\delta}\Prb_{\delta(t)}(\cE)
        \ge 2f(t)(1-f(t))\log n
    \end{split}
\end{equation*}
 Hence, we get
 \[f'(t)\ge  \min (C_d\lambda n^{d-3/2},2 f(t)(1-f(t))\log n)\]
 where $C_d$ depends on $d$.
 Up to increasing $\kappa$, we can assume that for 
$n\ge \kappa$, we have $C_d\lambda n^{d-3/2}\ge\log n$. Hence, for $n\ge \kappa$, we have
 \[f'(t)\ge  2 f(t)(1-f(t))\log n.\]
By integrating between $0$ and $t_0\in[0,\delta_1/2\kappa]$, we obtain
\begin{equation*}
    \log \left (\frac{f(t_0)}{1-f(t_0)}\frac{1-f(0)}{f(0)}\right) \ge2 t_0\log n.
\end{equation*}
It yields
\begin{equation*}
  \frac{ 1}{(1-f(t_0))f(0)}\ge \exp\left ( 2t_0\log n \right)
\end{equation*}
and
\begin{equation*}
  f(t_0)\ge 1-  n^\ep \exp\left ( -2t_0\log n \right).
\end{equation*}
Hence, for $n\ge \kappa$, we have
\begin{equation*}
  \Prb_{\delta_1-\ep \kappa}(\cE)\ge 1- \frac{1}{n^ \ep}.
\end{equation*}
The result follows.
 
\end{proof}

\section{Grimmett--Marstrand dynamic renormalization}\label{sec:GM}
In this section $\mu$ is a general distribution with a finite $d$-moment. We aim to prove Proposition \ref{prop:GM}.
For $B,E,F\subset \mathbb R^d$ and $n\ge 1$, denote
\[\cE_n(B,E,F):=\left\{\exists (x,r)\in \eta \cap(E\times[n,+\infty)):  B\stackrel{\cO(\eta\cap(F\times \sR))}{ \longleftrightarrow} \mathrm B_r^ x\right\}.\]
For short, we will write $\cE_n(E,F)$ instead of $\cE_n(\mathrm B_n,E,F)$.

For $n\ge 1$, we define $\Lambda_n:=[-n,n]^d$.
We will first need the following lemma.

\begin{lem}\label{lem:balltobox}There exists a positive constant $c'_d$ depending only on $d$ such that for $N\ge n\ge 1$, $\ep>0$ such that for any $\rho \in[1,\sqrt{d}+2]$ we have
\begin{equation*}
    \Prb_{\lambda,\mu}(\cE_n(\mathrm B_{(\rho + 1/2)N}\setminus \mathrm B_{\rho N},B_{(\rho + 1/2)N} )\ge 1-\ep,
\end{equation*}
then 
for any $z\in \mathbb R^d$ such that $N\le \|z\|_2\le (\sqrt d +2)N$, we have
\begin{equation*}
    \Prb_{\lambda,\mu}(\cE_n(z+\Lambda_N, \mathrm B_{3\sqrt{d} N}))\ge 1-\ep^{1/c'_d}.
\end{equation*}

\end{lem}

\begin{proof} Let $\rho\in[1,\sqrt{d}+2]$.
There exists a constant $c'_d$ depending only on $d$ and a family of balls $(\mathrm B_N^ {x_i})_{1\le i \le m_0}$ with $m_0\le c'_d $, such that for all $i\le m_0$, $x_i\in \partial \mathrm B_{\rho N}$ and  \[\mathrm B_{(\rho+1/2) N}\setminus\mathrm B_{\rho N}\subset \bigcup_{1\le i \le m_0}\mathrm B_N^ {x_i} .\]
It follows that
\[ \Prb_{\lambda,\mu}\left(\bigcup_{1\le i \le m_0}\cE_n(\mathrm B_N^ {x_i},\mathrm B_{(\rho+1/2)N})\right)\ge 1-\ep.\]
By invariance under rotation around $0$, and using the square root trick, we have
\begin{equation*}
    \Prb_{\lambda,\mu}(\cE_n(\mathrm B_N+\rho N \textbf{e}_1,\mathrm B_{(\rho+1/2)N}))\ge 1-\ep ^{1/c'_d}
\end{equation*}
where $\textbf{e}_1=(1,0,\dots,0)$.
Finally, we get for any $z\in \mathbb \sR^d$ such that $N\le \|z\|_2\le (\sqrt d +2)N$
\begin{equation*}
    \Prb_{\lambda,\mu}(\cE_n(z+\Lambda_N, \mathrm B_{3\sqrt{d} N}))\ge  \Prb_{\lambda,\mu}(\cE_n(\mathrm B_N+\|z\|_2\textbf{e}_1, \mathrm B_{\|z\|_2+N/2}))\ge 1-\ep ^{1/c'_d}.
\end{equation*}
\end{proof}
\subsection{Sprinkling}
Let $n\ge N$ and $\lambda>0$. 
For $A\subset \sR^d$, we denote by $\Delta A$ the set of all balls intersecting the boundary of $A$, that is,
\[\Delta A:=\left \{(x,r)\in \sR^ d\times \sR_+:\mathrm B_r^ x\cap \partial A \neq \emptyset\right\}.\] In the following sprinkling lemma we consider both the connection with a deterministic set (that correspond to the set $C$) and with a random set (balls in the set $B$).
\begin{lem}[Sprinkling lemma] \label{lem:sprinkling}Let $\xi>0$, $\beta>0$ and $\lambda>0$. Let $A\subset R\subset \sR^d$. Let $B\subset (R\times \sR_+)\setminus \Delta A$. Let $C\subset R\setminus A$. Let us assume that 
\begin{equation}\label{eq:assumptionconnection}
    \Prb_{\lambda,\mu}(A\stackrel {R}{\longleftrightarrow}\cO( B\cap \eta)\cup C)\ge 1-\exp(-3\lambda/\xi)
\end{equation}
then we have
\begin{equation*}
    \Prb(A\stackrel {\cO((\eta\cup \eta')\cap(R\times \sR_+))}{\longleftrightarrow}\cO(B\cap \eta)\cup C\,|\,\eta\cap \Delta A=\emptyset )\ge 1-\exp(-\beta)-\exp(-\lambda/\xi)
\end{equation*}
where $\eta$ (respectively $\eta'$) is a Poisson point process of intensity  $ \lambda dz\otimes \mu$ (respectively  $ \beta\xi dz\otimes \mu$).
\end{lem}
\begin{proof}[Proof of Lemma \ref{lem:sprinkling}]
To prove this lemma, it is more convenient to work in a discrete setting instead of the continuous setting. For that purpose, we introduce a very dense set of points in an enlarged space in a similar way as in \cite{ATT18}. Set $\overline \Omega=\sR^d\times \sR_+\times [0,1]$. Let $\overline \eta $ be a Poisson point process of intensity $m\lambda_1 dz\otimes\mu\otimes du$ where $\lambda_1=2(\lambda+\beta \xi)$.
On top of every point $z$ of $\overline \eta$, we add independently a uniform random variable $U_z$ on $[0,1]$.
Set \[\eta_\lambda :=\left\{(x,r): \exists u\le\frac{1}{m}\quad (x,r,u)\in \overline \eta,\quad U_{(x,r,u)}\le \frac{\lambda}{\lambda_1}\right\}.\]
It is easy to check that the set $\eta_\lambda$ is a Poisson point process of intensity  $ \lambda dz\otimes \mu$.
It follows that 
\begin{equation*}
    \E\left [ \Prb\Big(A\stackrel {\cO(\eta_\lambda\cap(R\times \sR_+))}{\longleftrightarrow}\cO(B\cap \eta_\lambda)\cup C\,|\,\overline\eta\Big)\right]=\Prb_{\lambda,\mu}(A\stackrel {R}{\longleftrightarrow}\cO(B\cap \eta)\cup C).
\end{equation*}
Denote $\pi$ the projection from $\overline \Omega$ to $\mathbb R^d\times\sR_+$
\[\pi(\overline \eta):=\{(x,r):(x,r,u)\in\overline \eta\}.\]
By similar argument than in \cite{ATT18}, we have
\begin{equation}\label{eq:vardense}
   \Var\left( \Prb\Big(A\stackrel {\cO(\eta_\lambda\cap(R\times \sR_+))}{\longleftrightarrow}\cO(B\cap \eta_\lambda)\cup C\,|\,\pi(\overline\eta)\Big)\right)\le \frac{1}{m}.
\end{equation}
Let $W$ be the sites in $(\Delta A\times [0,1/m])\cap \overline\eta$ that are connected to  $\cO(B\cap\eta_\lambda)\cup C$ in $\cO((\eta_\lambda\setminus \Delta A)\cap (R\times \sR_+))$, that is
\begin{equation*}
    W:=\left\{(x,r,u)\in \left(\Delta A\times \left[0,\frac{1}{m}\right]\right) \cap \overline\eta: \mathrm B_r^x \stackrel{\cO((\eta_\lambda\setminus \Delta A)\cap (R\times \sR_+)))}{\longleftrightarrow }\cO(B\cap\eta_\lambda)\cup C\right\}
\end{equation*}
Conditionally on $\pi(\overline\eta)$, the sites in $W$ are independent of $(U_z)_{z\in W}$.
We have
\begin{equation*}
  \Prb\left(A\stackrel {\cO(\eta_\lambda\cap(R\times \sR_+))}{\not\longleftrightarrow}\cO(B\cap\eta_\lambda)\cup C\,\Big|\,\pi(\overline\eta)\right)\ge(1-\lambda/\lambda_1)^{t-1}  \Prb(|W|< t \,|\,\pi(\overline\eta)).
\end{equation*}
Let $\cE_0$ be the following event 
\begin{equation*}
    \cE_0:=\left\{\eta_\lambda\cap\Delta A=\emptyset\right\}.
\end{equation*}
It yields that
\begin{equation*}
    \begin{split}
    \Prb&\Big(A\stackrel {\cO(\eta_{\lambda+\beta\xi}\cap(R\times \sR_+))}{\longleftrightarrow}\cO(B\cap\eta_\lambda)\cup C\,|\,\cE_0,\pi(\overline \eta)\Big)\\&\hspace{3cm}\ge\Prb(\exists z\in W\quad U_{z}\in (\lambda/\lambda_1, (\lambda+\beta \xi)/\lambda_1)], \,|W|\ge t\,|\,\pi(\overline \eta))\\
    &\hspace{3cm}\ge (1-(1-\beta \xi/\lambda_1)^t)\Prb(|W|\ge t\,|\, \pi(\overline \eta))\\
    &\hspace{3cm}\ge (1-(1-\beta \xi/\lambda_1)^t)\left(1-\frac{\Prb\Big(A\stackrel {\cO(\eta_\lambda\cap(R\times \sR_+))}{\not\longleftrightarrow}\cO(B\cap\eta_\lambda)\cup C\,|\,\pi(\overline \eta))\Big)}{(1-\lambda/\lambda_1)^{t-1}}\right).
    \end{split}
\end{equation*}
Finally by choosing $t= \lambda_1/ \xi$ and using $-2x\le \log(1-x)\le -x$ for $x\in[0,1/2)$, we get
\begin{equation*}
\begin{split}
     \Prb&\Big(A\stackrel {\cO(\eta_{\lambda+\beta\xi}\cap(R\times \sR_+))}{\longleftrightarrow}\cO(B\cap\eta_\lambda)\cup C\,|\,\cE_0,\pi(\overline \eta)\Big)\\&\hspace{2cm}\ge (1-\exp(-\beta))\Big(1-\exp(2\lambda/\xi)\Prb\Big(A\stackrel {\cO(\eta_\lambda\cap(R\times \sR_+))}{\not\longleftrightarrow}\cO(B\cap\eta_\lambda)\cup C\,|\,\pi(\overline \eta)\Big)\Big)\\&\hspace{2cm}\ge 1-\exp(-\beta)-\exp(2\lambda/\xi)\Prb\Big(A\stackrel {\cO(\eta_\lambda\cap(R\times \sR_+))}{\not\longleftrightarrow}\cO(B\cap\eta_\lambda)\cup C\,|\,\pi(\overline \eta)\Big).
     \end{split}
\end{equation*}
Using the assumption \eqref{eq:assumptionconnection} and using inequality \eqref{eq:vardense}, for every $m$ large enough such that $\Prb_{\lambda,\mu}(\eta\cap \Delta A=\emptyset)\ge 4/\sqrt{m}$, we have with positive probability using Markov inequality
\begin{equation*}
\begin{split}
     &  \frac{\Prb\Big(A\stackrel {\cO((\eta\cup \eta')\cap(R\times \sR_+))}{\longleftrightarrow}\cO(B\cap\eta)\cup C,\eta\cap \Delta A=\emptyset \Big)+4/\sqrt{m}}{\Prb(\eta\cap \Delta A=\emptyset)-4/\sqrt{m}}\\&\hspace{7cm
    }\ge 1-\exp(-\beta)-\exp(-\lambda/\xi)-\frac{4\exp(2\lambda/\xi)}{\sqrt{m}}
       \end{split}
\end{equation*}
where $\eta$ (respectively $\eta'$) is a Poisson point process of intensity  $ \lambda dz\otimes \mu$ (respectively  $ \beta\xi dz\otimes \mu$).
The result follows by letting $m$ go to infinity.
\end{proof}
\subsection {Dynamic exploration}In this section, we implement a Grimmett--Marstrand dynamic exploration to prove Proposition \ref{prop:GM}.
This section is inspired from \cite{DCKT}.
Fix $\beta\ge 0$ large enough such that
\[\exp(-\beta) \le \frac{1-p_c^{site}}{3}\]
and $\ep_0>0$ small enough such that
\[\ep_0^{1/3} \le \frac{1-p_c^{site}}{3}\]
where $p_c^{site}$ is the critical parameter of Bernoulli site percolation on $\sZ^2$. Thanks to this choice, we have
\begin{equation}\label{eq:pcsite}
    1-\exp(-\beta)-\ep_0^{1/3}>p_c^{site}.
\end{equation}
Let $c'_d$ be the constant of Lemma \ref{lem:balltobox}.
Let us assume that there exists $\ep\le \ep_0^{c'_d}$ and $n\le N$ such that
we have
for every $\rho \in[1,\sqrt{d}+2]$
\begin{equation*}
    \Prb_{\lambda,\mu}(\cE_n(\mathrm B_{(\rho + 1/2)N}\setminus \mathrm B_{\rho N},\mathrm B_{(\rho + 1/2)N} )\ge 1-\ep.
\end{equation*}
From now on, $\ep,n,N$ are fixed such that the previous inequality holds.
By Lemma \ref{lem:balltobox},
for all $z\in \mathbb \sR^d$ such that $N\le \|z\|_2\le (\sqrt d +2)N$, we have
\begin{equation}\label{eq:seedgoodproba}
    \Prb_{\lambda,\mu}(\cE_n(z+\Lambda_N,\mathrm B_{3\sqrt{d} N}))\ge 1-\ep^{1/c'_d}\ge 1-\ep_0.
\end{equation}
For every $x\in\sZ^2$, set \[\Lambda_x :=2Nx+\Lambda_N\subset\sR^ d\qquad\text{and}\qquad\widetilde B_x:=2Nx+\mathrm B_{4\sqrt{d} N}\subset\sR^ d\]
where we identify $x=(x_1,x_2)$ with $(x_1,x_2,0,\dots,0)\in\sZ^d$.
Let $\eta$ be a Poisson point process of intensity $\lambda dz\otimes\mu$.
Let $(\eta^ x, x\in \sZ^2)$ be a family of independent Poisson point process of intensity $\beta \xi dz \textbf{1}_{\widetilde B_x}\otimes \mu$ where $\xi$ is such that
$\ep^{1/c'_d}= \exp(-3\lambda /\xi)$.
We aim to prove that the box $\Lambda_0$ intersects an infinite connected component of balls with positive probability in $\cO(\eta\cup_{x\in\sZ^2}\eta^x)$.
For a configuration $\eta$ we denote by $\sC(\eta)$ the set of balls in $\eta$ that are connected to $\mathrm B_n$ in $\cO(\eta)$, that is,
\[\sC(\eta):=\bigcup \left\{\mathrm B_r^x: (x,r)\in \eta, \mathrm B_n \stackrel{\cO(\eta)}{\longleftrightarrow}\mathrm B_r^x\right\}.\]
Set $X_0:=(A_0,B_0)=(\{0\},\emptyset)$, $\mathcal B_0:=\mathrm B_n$ and $\eta_0:=\eta$. Let $t\ge 0$. We build $\eta_{t+1}$, $\mathcal B_{t+1}$ and $X_{t+1}$ from $\eta_t$, $\mathcal B_t$ and $X_t$ as follows. 
If there is no edge in $(\sZ^2,\E^2)$ connecting $A_t$ to $(A_t\cup B_t) ^c$, then set $X_{t+1}:=X_t$, $\mathcal B_{t+1}:=\mathcal B_{t}$ $\eta_{t+1}:=\eta_t$.
Otherwise, let us denote by $x_{t}$ the extremity in $(A_t\cup B_t)^c$ of this edge (if there are several such edges, we choose one using a deterministic rule) and let $t_0< t$ be such that $x_{t_0}$ is the other extremity of the edge. By definition, we have $x_{t_0}\in A_t$.
Define
\[\eta_{t+1}:=\eta_t \cup \eta^{x_t}\]
and
\[(A_{t+1},B_{t+1}):=\left\{\begin{array}{ll} (A_t\cup\{x_t\},B_t)&\mbox {if the event $\cE_n(\mathcal B_{t_0},\Lambda_{x_t},\widetilde B_{x_t})$ occurs for $\eta_{t+1}$} \\
(A_t,B_t\cup\{x_t\})&\mbox{otherwise.}
\end{array}\right.\]
If $x_t\in A_{t+1}$, to define $\mathcal B_t$, we pick according to some deterministic rule a ball in
\[\left\{\mathrm B_r^ x: \,(x,r)\in \eta_{t+1} \cap(\Lambda_{x_t}\times[n,+\infty)), \, \mathcal B_{t_0}\stackrel{\cO(\eta_{t+1}\cap(\widetilde B_{x_t}\times \sR_+))}{ \longleftrightarrow} \mathrm B_r^ x\right\}.\]
Note that \[\eta_\infty:=\bigcup_{t\ge 0}\eta_t \subset \left(\eta\cup \bigcup_{x\in\sZ^2}\eta^x \right)\] and if the exploration never ends then the ball $\mathrm B_n$ is connected to infinity in $\cO(\eta_\infty)$.
We aim to prove that 
\[\Prb(B_{t+1}=B_t|X_0,\dots,X_t)\ge 1-\exp(-\beta)-\ep ^{1/(3c'_d)}>p_c^{site}.\]
Let us condition on $\sC(\eta_0),\dots,\sC(\eta_t)$.
Let $x_t$ be as defined and $x_{t_0}$ corresponding to the other endpoint. By construction $x_{t_0}\in A_t$ and $\mathcal B_{t_0}$ is connected to  $\mathrm B_n $ in $\cO(\eta_t)$, that is  $\mathcal B_{t_0}\subset \sC(\eta_t)$. It follows that 
\begin{equation*}
\begin{split}
    \Prb&( \cE_n(\mathcal B_{t_0},\Lambda_{x_t},\widetilde B_{x_t}) \text{ occurs for $\eta_{t+1}$}|\sC(\eta_0),\dots,\sC(\eta_t)) 
    \\&\hspace{5cm}=\Prb( \cE_n(\sC(\eta_t),\Lambda_{x_t},\widetilde B_{x_t}) \text{ occurs for $\eta_{t+1}$}|\sC(\eta_t)) .
\end{split}
\end{equation*}
If $ \cE_n(\mathcal B_{t_0},\Lambda_{x_t},\widetilde B_{x_t})$  occurs for $\eta_{t}$ then we are done and $x_t\in \cA_{t+1}$. Otherwise, we want to do a sprinkling to get another chance for this event to occur. 
Let $z\in \Lambda_{x_{t_0}}$. Since $\|x_{t_0}-x_t\|_\infty =1$, it follows that
\[N\le \|z-2Nx_t\|_2\le (\sqrt d +2)N.\]
Note that $\sC(\eta_t)$ contains a ball of radius at least $n$ centered in $\Lambda_{x_{t_0}}$, thanks to inequality \eqref{eq:seedgoodproba}, it follows that
\[\Prb(\cE_n(\sC(\eta_t),\Lambda_{x_t},\widetilde B_{x_t})\text{ occurs for $\eta'$}|\sC(\eta_t))\ge 1-\ep^{1/c'_d}\]
where $\eta'$ is an independent Poisson point process of intensity $\lambda dz\otimes \mu$.
Note that when we condition on $\sC(\eta_t)$ the points in $\eta_t$ such that the corresponding ball does not intersect $\sC(\eta_t)$ are unexplored and $\eta_t \cap \Delta \sC(\eta_t)=\emptyset$.
\begin{equation*}
\begin{split}
    \Prb&(\cE_n(\sC(\eta_t),\Lambda_{x_t},\widetilde B_{x_t})\text{ occurs for $\eta_{t+1}$}|\sC(\eta_t))\\
    &\qquad=\Prb(\cE_n(\sC(\eta_t),\Lambda_{x_t},\widetilde B_{x_t})\text{ occurs for $\eta'\cup\eta^ {x_t}$}|\,\eta' \cap \Delta \sC(\eta_t)=\emptyset,\sC(\eta_t))
    \end{split}
\end{equation*}
where $\eta'$ is an independent Poisson point process of intensity $\lambda dz\otimes \mu$.
By using the Lemma \ref{lem:sprinkling} conditioning on $\sC(\eta_t)$ (for $A=\sC(\eta_t) $, $B=(\Lambda_{x_t}\times[n,+\infty))\setminus \Delta \sC(\eta_t)$ and $C=\emptyset$) and inequality \eqref{eq:pcsite}, it follows that 
\[ \Prb( \cE_n(\mathcal B_{t_0},\Lambda_{x_t},\widetilde B_{x_t}) \text{ occurs for $\eta_{t+1}$}|\sC(\eta_0),\dots,\sC(\eta_t))\ge 1-\exp(-\beta)-\ep^{1/(3c'_d)}>p_c^{site}.\]
We need the following lemma to conclude. 
\begin{lem}[Lemma 9 in \cite{DCKT}]\label{lem:explor}Let $X_t = (A_t, B_t)$ be a random exploration sequence and assume that there exists some $q >p_c^{site}$ such that for every $t\ge 0$
\[\Prb(B_{t+1}=B_t|X_0,\dots,X_t)\ge q\quad \text{a.s.,}\]
then the process X percolates with probability larger than a constant $c = c(q) >0$ that
can be taken arbitrarily close to 1 provided that q is close enough to 1.
\end{lem}
We can deduce that with positive probability, there exists an infinite connected component in $\cO(\eta_{\infty})$. Note that there exists a constant $\kappa_d$ depending only on the dimension $d$ such that $\eta_\infty$ is stochastically dominated by a Poisson point process of intensity $(\lambda+\kappa_d\beta \xi)dz\otimes \mu$. Hence,
\[\lambda\left(1+ \frac{3\kappa_dc'_d\beta}{-\log \ep}\right)=\lambda+\kappa_d\beta \xi\ge \lambda_c(\mu).\]
The result follows.

\section{Supercritical sharpness: proof of Theorem \ref{thm:main2}}\label{sec:main2}
In this section, we prove Theorem \ref{thm:main2} by adapting the seedless Grimmett--Marstrand strategy in \cite{DCKT} to the Boolean context. We refer to \cite{DCKT} as most of the arguments will only be sketched.
In what follows, we fix $\mu$ a distribution on the radius that has a finite $d$-moment.
Let $\eta$ be a Poisson point process of intensity $\lambda dz\otimes \mu$.
Let $k\le K$. Denote by $A_2(k,K)$ the event where there exists at least two disjoint connected components in $\cO(\eta \setminus \cB(\Lambda_K))\cap \Lambda_K$ that intersect both $\Lambda_k$ and $\partial \Lambda_K$ where
\[\cB(\Lambda_K):=\left\{(x,r)\in(\sR^d\times\sR_+)\setminus (\Lambda_{2K}\times [0,K]): \mathrm B_r^x\cap \Lambda_K\ne\emptyset\right\}.\]
The set $\cB(\Lambda_K)$ may be thought as a set of bad balls that are typically not present when $K$ is large. We will need the following claim that follows from similar computations as in the inequality \eqref{eq:interbigball}.
\begin{claim}\label{claim:badballs}For all $\lambda>0$. We have 
\[\lim_{K\rightarrow\infty}\Prb_{\lambda,\mu}(\eta\cap\cB(\Lambda_K)\ne\emptyset)=0.\]
\end{claim}

\begin{prop}\label{prop:connectionset}There exists some $C>0$ and $\ep_0>0$ (depending on $d$ only) such that the following holds.
Let $\ep\in(0,\ep_0)$ and let 
$1\le k\le K\le n\le N$ such that $K\le \ep ^2 n $ and
\begin{enumerate}[(a)]
\item\label{cond:a} $\Prb_{\lambda,\mu}\left( 0\stackrel{\cO(\eta \setminus (\cB(\Lambda_K)\cup\cB(\Lambda_N)))}{\longleftrightarrow}\partial \Lambda_ N\right)\ge \ep$ and $\Prb_{\lambda,\mu}\left( 0\stackrel{\cO(\eta \setminus \cB(\Lambda_K))}{\longleftrightarrow}\partial \Lambda_ K\right)\ge \ep$,
\item\label{cond:b} $\Prb_{\lambda,\mu}\left( \Lambda_k\stackrel{\cO(\eta \setminus (\cB(\Lambda_K)\cup\cB(\Lambda_N)))}{\longleftrightarrow}\partial \Lambda_ N\right)\ge 1-\exp\left(-\frac{1}{\ep}\right)$,
\item\label{cond:c} $\Prb_{\lambda,\mu}(A_2(k,K))\le \exp\left(-\frac{1}{\ep}\right)$ and $\Prb_{\lambda,\mu}(A_2(n,N))\le \exp\left(-\frac{1}{\ep}\right)$.
\end{enumerate}
Then, we have
\[\lambda+C\ep\ge \lambda_c^{S_{4dN}}(\mu|_{[0,N]}).\]
\end{prop}
Let us first explain how we can deduce Theorem \ref{thm:main2} from Proposition \ref{prop:connectionset}.
We will need the following lemma which is a consequence of the work of Meester and Roy \cite{MeesterRoy94Uniqueness} who proved uniqueness of the infinite cluster in Boolean percolation.
\begin{lem}\label{lem:twoarm} Let $\lambda>\lambda_c(\mu)$.
We have
\[\lim_{n\rightarrow\infty }\lim_{N\rightarrow\infty }\Prb_{\lambda,\mu}(A_2(n,N))=0.\]
\end{lem}

\begin{proof}[Proof of Theorem \ref{thm:main2}]
Let $\lambda>\lambda_c(\mu)$. Let $\ep_0$ and $C$ be as in the statement of Proposition \ref{prop:connectionset}. There exists $\ep\in(0,\ep_0)$ such that
\[\inf_{n\ge 1}\Prb_{\lambda,\mu}(0\longleftrightarrow\partial \Lambda_n)\ge 2\ep.\]
Moreover, thanks to Claim \ref{claim:badballs}, Lemma \ref{lem:twoarm} and the existence of an infinite cluster in the supercritical regime, there exists $n_0\ge 1$ large enough such that
\[\forall n\ge n_0\qquad \Prb_{\lambda,\mu}(\eta\cap\cB(\Lambda_n)\ne\emptyset)\le \frac{1}{4}\exp\left(-\frac{1}{\ep}\right),\]
\[\forall n\ge n_0 \qquad \inf_{N\ge n_0}\Prb_{\lambda,\mu}(A_2(n,N))\le \frac{1}{2}\exp\left(-\frac{1}{\ep}\right)\]
and
\[\forall n\ge n_0\qquad \inf_{N\ge n }\Prb_{\lambda,\mu}(\Lambda_{n}\longleftrightarrow\partial \Lambda_N)\ge 1-\frac{1}{2}\exp\left(-\frac{1}{\ep}\right).\]
Let $k=n_0$, $K$ be the first integer such that 
\[\Prb_{\lambda,\mu}(A_2(k,K))\le \exp\left(-\frac{1}{\ep}\right),\]
let $n$ be such that $n\ge2 K/\ep^2$ and $N$ be the first integer such that
\[\Prb_{\lambda,\mu}(A_2(n,N))\le \exp\left(-\frac{1}{\ep}\right).\]
It follows that the conditions \ref{cond:a}, \ref{cond:b} and \ref{cond:c} hold and by Proposition \ref{prop:connectionset}, there exists $C$ depending only on $d$ such that
\[\lambda+C\ep\ge \lambda_c^{S_{4dN}}(\mu|_{[0,N]}).\]
It follows that 
\[\lambda+C\ep\ge \inf_{N\ge 1}\lambda_c^{S_{N}}(\mu|_{[0,N]})\ge \lambda_c(\mu).\]
The result follows by letting $\ep$ go to $0$ and $\lambda$ go to $\lambda_c(\mu)$.
\end{proof}

Let us now sketch the proof of Proposition \ref{prop:connectionset}.
The following lemma is proved in the context of percolation in \cite{DCKT} it is a key part to prove Proposition \ref{prop:connectionset}. 
\begin{lem}\label{lem:connectionset}Assume that the conditions \ref{cond:a}, \ref{cond:b} and \ref{cond:c} hold. Then there exists some $c>0$ (depending on $d$ only) such that for every connected set $S\ni 0$ with a diameter larger than
$n$,
\[\Prb_{\lambda,\mu|_{[0,N]}}(S\longleftrightarrow F(N))\ge 1-2 \exp\left(-\frac{c}{\ep}\right)\]
where $F(N) := \{(x_1,\dots,x_d)\in\partial \Lambda_N : x_1 = N, x_2\ge0,\dots,x_d\ge 0\}$.
\end{lem}

\begin{proof}[Sketch of the proof of Lemma \ref{lem:connectionset}] We build $x_1,\dots,x_l\in S$ with $l\ge c_1/\ep^2$ in the same way as in the proof of Lemma 9 in \cite{DCKT} with the additional condition that the boxes $x_i+\Lambda_{2K}$ are disjoint. For an event $E$, $x+E$ denote the event translated by $x$.
Set 
\[E_i:=\left\{x_i\stackrel{\cO(\eta  \setminus (x_i+\cB(\Lambda_K)))}{\longleftrightarrow}x_i+\partial\Lambda_ K\right\}\cap(x_i+A_2(k,K))^ c\]
and
\[B_i:=\left\{x_i+\Lambda_k\stackrel{\cO(\eta \setminus (x_i+\cB(\Lambda_K)\cup\cB(\Lambda_N)))}{\longleftrightarrow}\partial \Lambda_ N\right\}^c.\]
The events $E_i$ are independent since the event $E_i$ only depends on the Poisson point process inside $x_i+\Lambda_{2K}$. 
If the event $E_i\setminus B_i$ occurs for some $i$,
then there exists a path from $x_i$ to $x_i+\partial \Lambda_K$ and a path from $x_i+\Lambda_k$ to $\partial \Lambda_N$ and these two paths do not use balls in $x_i+\cB(\Lambda_K)$. On the event $(x_i+A_2(k,K))^ c$, they must intersect inside $ x_i+\Lambda_K$. This creates a path from $x_i$ to $\partial \Lambda_N$.
Note that the square root-trick holds for Boolean percolation because we have FKG and the required symmetry.
We conclude similarly.
\end{proof}

\begin{proof}[Proof of Proposition \ref{prop:connectionset}]

The proof uses a renormalization and a sprinkling argument (similar to Section \ref{sec:GM}, at the exception that here we do not use a big ball as a seed but the path itself).

Let $q$ be large enough such that $c(q)\ge 1/2$ where $c(q)$ is defined in Lemma \ref{lem:explor}.
Fix $\beta>0$ large enough such that
\[\exp(-\beta)\le \frac{1-q}{3}.\]
Fix  $\ep_0$ small enough depending on $d$ and $q$ such that
\[2\exp\left(-\frac{c}{3\ep_0}\right)\le \frac{1-q}{3}\]
where $c$ depends only on $d$ is as in the statement of Lemma \ref{lem:connectionset}.
Thanks to this choice, we have
\begin{equation}\label{eq:condsprink}
    1-\exp(-\beta)-2\exp\left(-\frac{c}{3\ep_0}\right)\ge q.
\end{equation}
Let $\ep\in(0,\ep_0)$. Let 
$1\le k\le K\le n\le N$ such that $K\le \ep ^2 n $ and the conditions \ref{cond:a}, \ref{cond:b} and \ref{cond:c} hold.

For every $x\in\sZ^2$, set \[\Lambda_x =2Nx+\Lambda_N\subset\sR^ d\qquad\text{and}\qquad\widetilde B_x =2Nx+\mathrm B_{4\sqrt{d} N}\subset\sR^ d\]
where we identify $x=(x_1,x_2)$ with $(x_1,x_2,0,\dots,0)\in\sZ^d$.

Let $\eta$ be a Poisson point process of intensity $\lambda \mathbf{1}_{S_{4dN}}dz\otimes\mu$.  
Let $(\eta^ x, x\in \sZ^2)$ be a family of independent Poisson point process of intensity $\beta \ep dz \textbf{1}_{\widetilde B_x}\otimes \mu$.
We aim to prove that the box $\Lambda_0$ intersects an infinite component of balls with positive probability in $\cO(\eta\cup_{x\in\sZ^2}\eta^x)$.
For a configuration $\eta$ we denote by $\sC(\eta)$ the cluster of $0$ in $\eta$, that is,
\[\sC(\eta):=\bigcup \left\{\mathrm B_r^x: (x,r)\in \eta, \,0 \stackrel{\cO(\eta)}{\longleftrightarrow}\mathrm B_r^x\right\}.\]
Set $X_0:=(A_0,B_0)=(\{0\},0)$ and $\eta_0:=\eta\cap (S_{4dN}\times[0,N])$. Let $t\ge 0$. We build $\eta_{t+1}$ and $X_{t+1}$ from $\eta_t$ and $X_t$ as follows. 
If there is no edge in $(\sZ^2,\E^2)$ connecting $A_t$ to $(A_t\cup B_t) ^c$, then set $X_{t+1}:=X_t$ and $\eta_{t+1}:=\eta_t$.
Otherwise, let us denote by $x_{t}$ the extremity in $(A_t\cup B_t)^c$ of this edge (if there are several such edges, we choose one using a deterministic rule).
Define
\[\eta_{t+1}:=\eta_t \cup \eta^{x_t}\]
and
\[(A_{t+1},B_{t+1}):=\left\{\begin{array}{ll} (A_t\cup\{x_t\},B_t)&\mbox {if $ 0 \stackrel{\cO(\eta_{t+1})}{\longleftrightarrow}\Lambda_{x_t}$} \\
(A_t,B_t\cup\{x_t\})&\mbox{otherwise.}
\end{array}\right.\]
Note that \[\eta_\infty:=\bigcup_{t\ge 0}\eta_t \subset \left(\eta_0\cup \bigcup_{x\in\sZ^2}\eta^x \right)\] and if the exploration never ends then $0$ is connected to infinity in $\cO(\eta_\infty\cap S_{4dN})$.
We aim to prove that 
\begin{equation}\label{eq:goalperco}
    \Prb\left(X\text { percolates }|\,0\stackrel{\cO(\eta_{0})}{\longleftrightarrow}\partial\Lambda_{0}\right)\ge \frac{1}{2}.
\end{equation}
The conditioning ensures that $\sC(\eta_0)$ has a diameter at least $N$. Note that there exists a constant $C$ depending only on the dimension $d$ such that $\eta_\infty$ is stochastically dominated by a Poisson point process of intensity $(\lambda+C\ep)dz\otimes \mu$. If inequality \eqref{eq:goalperco} is true then $\lambda+C\ep\ge \lambda_c^ {S_{4dN}}(\mu|_{[0,N]})$ and the result follows. Let us now prove that inequality \eqref{eq:goalperco} holds.
Let us condition on $\sC(\eta_0),\dots,\sC(\eta_t)$.
We have
\begin{equation*}
\begin{split}
     \Prb&\left(\Lambda_0\stackrel{\cO(\eta_{t+1})}{\longleftrightarrow} \Lambda_{x_t}|\,\sC(\eta_0),\dots,\sC(\eta_t), \partial \Lambda_0\cap \sC(\eta_0)\ne\emptyset \right)\\
    &\hspace{5cm}=\Prb\left(\Lambda_0\stackrel{\cO(\eta_{t+1})}{\longleftrightarrow} \Lambda_{x_t}|\,\sC(\eta_t),  \partial \Lambda_0\cap \sC(\eta_t)\ne\emptyset\right) .
\end{split}
\end{equation*}
On the conditioning by $\{ \partial \Lambda_0\cap \sC(\eta_t)\ne\emptyset\}$, the cluster $\sC(\eta_t)$ has a diameter at least $N$. Using Lemma \ref{lem:connectionset}, we have
\[\Prb(\sC(\eta_t)\longleftrightarrow\Lambda_{x_t}\text{ occurs for $\eta'$} |\,\sC(\eta_t),  \partial \Lambda_0\cap \sC(\eta_t)\ne\emptyset)\ge 1-2 \exp\left(-\frac{c}{\ep}\right)\,.\]
where $\eta'$ is an independent Poisson point process of intensity $\lambda dz\otimes \mu$.
We conclude that inequality \eqref{eq:goalperco} holds using the sprinkling Lemma \ref{lem:sprinkling}, Lemma \ref{lem:explor} and inequality \eqref{eq:condsprink} (see the proof of Proposition \ref{prop:GM} in Section \ref{sec:GM}).
\end{proof}

\begin{proof}[Proof of Lemma \ref{lem:twoarm}]
Define $\cE_{n,N}$ to be the event that there exists at least two disjoint connected components in $\cO(\eta )\cap \Lambda_N$ that intersects both $\Lambda_n$ and $\partial \Lambda_N$.
We have
\begin{equation}\label{eq:a2}
    \Prb_{\lambda,\mu}(A_2(n,N))\le \Prb_{\lambda,\mu}(\cE_{n,N})+\Prb_{\lambda,\mu}(\eta\cap\cB(\Lambda_N)\ne\emptyset).
\end{equation}
Denote by $\cF$ the event that there exists at least two infinite clusters. By \cite{MeesterRoy94Uniqueness}, this event has probability $0$.
We have
\begin{equation*}
   \cF= \bigcup_{n\ge 1}\bigcap_{N\ge n}\cE_{n,N}
\end{equation*}
It follows that
\[\lim_{n\rightarrow\infty }\lim_{N\rightarrow\infty }\Prb_{\lambda,\mu}(\cE_{n,N})=0\,.\]
By Claim \ref{claim:badballs} and inequality \eqref{eq:a2}, it follows that
\[\lim_{n\rightarrow\infty }\lim_{N\rightarrow\infty }\Prb_{\lambda,\mu}(A_2(n,N))=0\,.\]
This yields the result.
\end{proof}

\appendix
\section{Inhomogeneous Talagrand formula}
In this section, we prove the inhomogeneous Talagrand formula (Proposition \ref{lem:inhomogeneous Talagrand}),
We will use the original Talagrand formula.
\begin{thm}[Talagrand formula \cite{Talagrand}]\label{thm: Talagrandformula}There exists a universal constant $C>0$ such that for any $p\in(0,1)$ for any function $f:\{0,1\}^N \rightarrow\{0,1\}$.
We have 
\[\sum_{i=1}^N \Inf_i(f)\ge \frac{C}{\log (\frac{2}{p(1-p)})}\Var(f)\log \left(\frac{1}{\max _{1\le i \le N}\Inf_i(f)}\right)\,\]
where
\[\Inf_i(f):=\mathrm P_{p}^{\otimes n}(f\circ\tau_{i}\ne f).\]
\end{thm}
\begin{proof}[Proof of Proposition \ref{lem:inhomogeneous Talagrand}] Fix $N\ge 1$. Fix $f:\{0,1\}^N\rightarrow \{0,1\}$ a boolean function. It is easy to check that the function $\Inf_i(f)$ is a polynomial function of the family $(p_1,\dots,p_N)$. As a result, the map 
$(p_1,\dots,p_N)\mapsto (\Inf_1 ^{p_1,\dots,p_N}(f),\dots,\Inf_N ^{p_1,\dots,p_N}(f))$ is continuous.
Hence, it is enough to check that the inequality holds only for $p_1,\dots,p_N$ that belong to the set $$E:=\{m2^{-\ell}: \ell\ge 1, \,m\le 2^\ell-1\}$$ which is dense in $[0,1]$.

From now on, we will assume that $p_1,\dots,p_N\in E$ and write $p_i=m_i 2^{-\ell_i}$ for all $1\le i \le N$.
Let $1\le i \le N$. We set 
\[\mathrm P:=\otimes_{i=1}^N\mathrm P_{1/2}^{\otimes \ell_i}.\] Let $(X_{i,1},\dots,X_{i,\ell_i})$ be $\ell_i$ independent Bernoulli of parameter $1/2$.
Let us define 
\[\pi(X_{i,1},\dots,X_{i,\ell_i}):=\sum_{k=1}^{\ell_i}\frac{X_{i,k}}{2^{k}}.\]
It is easy to check that $\pi(X_{i,1},\dots,X_{i,\ell_i})$ is uniform on $\{k2^{-\ell_i}: 0\le k\le 2^{\ell_i}-1\}$ and 
\[\mathrm P(\pi(X_{i,1},\dots,X_{i,\ell_i})\ge 1- m_i2^{-\ell_i})=m_i2^{-\ell_i}.\]
Set $Y_i:=\mathbf{1}_{\pi(X_{i,1},\dots,X_{i,\ell_i})\ge 1- p_i}$.
It follows that 
$(Y_i)_{1\le i \le N}$ has the distribution of $\otimes_{i=1}^N \mathrm {Ber}(p_i)$.
Set \[\widetilde f (X_{1,1},\dots, X_{1,\ell_1},\dots,X_{N,1},\dots,X_{N,\ell_N}):= f(Y_1,\dots,Y_N).\]
It is easy to check that $\Var(f)=\Var(\widetilde f)$.
Let us now compute the influence of $X_{i,j}$ on $Y_i$ for $1\le j \le \ell_i$.
We define $\sigma_j ^1$ (respectively $\sigma_j^0$ the function that for $x=(x_1,\dots,x_{\ell_i})\in\{0,1\}^{\ell_i}$ changes $x_j$ into $1$ (respectively into $0$).
Set \[\Delta_j \pi=\pi\circ \sigma_j ^1-\pi\circ \sigma_j ^0.\]
We have
\[\Inf_{i,j}(\widetilde f)= \mathrm P(\Delta_j\pi(X_{i,1},\dots,X_{i,\ell_i})\ne 0) \Inf_i(f).\]
Let us now compute $ \mathrm P(\Delta_j\pi(X_{i,1},\dots,X_{i,\ell_i})\ne 0)$.
We have 
\[\left\{\Delta_j\pi(X_{i,1},\dots,X_{i,\ell_i})\ne 0\right\}= \left\{\sum_{k=1, k\ne j}^{\ell_i}\frac{X_{i,k}}{2^{k}}\le 1- p_i, \, \sum_{k=1, k\ne j}^{\ell_i}\frac{X_{i,k}}{2^{k}}+\frac{1}{2^j}\ge 1-p_i\right\}.\]
Let $j_i\ge 1$ be the largest integer such that \[2^{-j_i}\ge p_i.\]
We distinguish two cases.
Let us first assume that $j\ge j_i$.
\[\{\Delta_j\pi(X_{i,1},\dots,X_{i,\ell_i})\ne 0\}\subset  \left\{\sum_{k=1}^{j-1}\frac{X_{i,k}}{2^{k}}\le 1- p_i, \, \sum_{k=1}^{j-1}\frac{X_{i,k}}{2^{k}}\ge 1-p_i- \frac{1}{2^{j-1}}\right\}\]
It follows that 
\[\mathrm P(\Delta_j\pi(X_{i,1},\dots,X_{i,\ell_i})\ne 0)\le \frac{1}{2^{j-1}}.\]
This rough upper-bound also holds for $ j<j_i$, but we aim here at obtaining a better upper-bound in that case.
Let us now assume that $j< j_i$. We have
\begin{equation}\label{eq:inclusion}
\{\Delta_j\pi(X_{i,1},\dots,X_{i,\ell_i})\ne 0\}=  \left\{\sum_{k=1}^{j-1}\frac{X_{i,k}}{2^{k}}\le 1- p_i, \, \sum_{k=1}^{j-1}\frac{X_{i,k}}{2^{k}}\ge 1-p_i- \frac{1}{2^{j}}-\frac{1}{2^j}\sum_{k=1}^{\ell_i-j}\frac{X_{i,k+j}}{2^k}\right\}.
\end{equation}
Note that we have almost surely
\[\sum_{k=1}^{j-1}\frac{X_{i,k}}{2^{k}}\le 1-\frac{1}{2^{j-1}}.\]
Hence, for the event on the right hand side of \eqref{eq:inclusion} to be non-empty, we need that 
\[1-p_i- \frac{1}{2^{j}}-\frac{1}{2^j}\sum_{k=1}^{\ell_i-j}\frac{X_{i,k+j}}{2^k}\le  1-\frac{1}{2^{j-1}}.\]
That is
\[\sum_{k=1}^{\ell_i-j}\frac{X_{i,k+j}}{2^k}\ge 1- 2^j p_i.\]
It follows that
\begin{equation*}
\begin{split}
\mathrm P&(\Delta_j\pi(X_{i,1},\dots,X_{i,\ell_i})\ne 0)\\
&\qquad\le \mathrm P\left(1-p_i-\frac{1}{2^{j-1}}\le\sum_{k=1}^{j-1}\frac{X_{i,k}}{2^{k}}\le 1- p_i\right) \mathrm P\left(\sum_{k=1}^{\ell_i-j}\frac{X_{i,k+j}}{2^k}\ge 1- 2^j p_i\right)\\
&\qquad\le\frac{1}{2^{j-1}}2^j p_i\le 2 p_i.
\end{split}
\end{equation*}
Finally, we have
\begin{equation*}
\sum_{j=1}^{\ell_i}\mathrm P(\Delta_j\pi(X_{i,1},\dots,X_{i,\ell_i})\ne 0)\le \sum_{j=1}^{j_i-1}2 p_i +\sum_{j=j_i}^{\ell_i}\frac{1}{2^{j-1}}\le 4 p_i |\log p_i|
\end{equation*}
and
\[\sum_{j=1}^{\ell_i}\Inf_{i,j}(\widetilde f)\le 4 p_i |\log p_i|\Inf_i(f).\]
Let us now apply Talagrand formula (Theorem \ref{thm: Talagrandformula})  to $\widetilde f$, we obtain
\[\sum_{i=1}^N\sum_{j=1}^{\ell_i}\Inf_{i,j}(\widetilde f)\ge C\Var(\widetilde f)\log \left(\frac{1}{\max\{ \Inf_{i,j}(\widetilde f): 1\le i \le N , 1\le j \le \ell_i\}}\right)\]
and
\[\sum_{i=1}^Np_i|\log p_i|\Inf_{i}(f)\ge C\Var(f)\log \left(\frac{1}{\max\{ p_i\Inf_{i}(f): 1\le i \le N \}}\right).\]
The result follows.

\end{proof}
\bibliographystyle{plain}
\bibliography{biblio}
\end{document}